\documentclass[a4paper]{gtart}

\usepackage{psfrag, graphicx, epstopdf, pstricks, subfigure, epsfig,color}

\usepackage{amsmath, amssymb, latexsym, euscript}

\usepackage{mathptmx, wrapfig}

\usepackage[inner=30mm, outer=30mm, textheight=225mm]{geometry}

\usepackage{mathrsfs,url}

\def\tri{\mathcal{T}}

\def\bkR{{\rm I\kern-.17em R}}
\def\R{\bkR}
\def\bkZ{{\rm Z\kern-.26em Z}}

\DeclareMathOperator{\e}{e}

\def\Quad{q}

\DeclareMathOperator{\wt}{wt}

\theoremstyle{plain}
\newtheorem{theorem}{Theorem}
\newtheorem*{theorem*}{Theorem}
\newtheorem{lemma}[theorem]{Lemma}
\newtheorem{proposition}[theorem]{Proposition}
\newtheorem{corollary}[theorem]{Corollary}

\newtheorem*{claim*}{Claim}

\theoremstyle{definition}
\newtheorem{definition}[theorem]{Definition}
\newtheorem*{definition*}{Definition}

\newtheorem{question}[theorem]{Question}

\theoremstyle{remark}
\newtheorem{remark}[theorem]{Remark}
\newtheorem{example}{Example}

\numberwithin{equation}{section}

\begin{document}

\title{Bounds for the genus of a normal surface}
\author{William Jaco, Jesse Johnson, Jonathan Spreer and Stephan Tillmann}
\begin{abstract}
This paper gives sharp linear bounds on the genus of a normal surface in a triangulated compact, 
orientable 3--manifold in terms of the quadrilaterals in its cell decomposition---different bounds arise from varying  hypotheses on the surface or triangulation.
 Two
applications of these bounds are given. First, the minimal triangulations of the product of a closed 
surface and the closed interval are determined. Second, an alternative approach 
of the realisation problem using normal surface theory is shown to be less powerful than 
its dual method using subcomplexes of polytopes.
\end{abstract}

\makeshorttitle



\section{Introduction}

The theory of normal surfaces, introduced by Kneser 
\cite{Kneser29ClosedSurfIn3Mflds} and further developed 
by Haken \cite{Haken61TheoNormFl, Haken62HomeomProb3Mflds}, 
plays a crucial role in $3$--manifold topology. 
Normal surfaces allow topological problems to be translated into algebraic problems or linear programs, and they are the
key to many important advances over the last $50$ years, including the solution of the unknot 
recognition problem by Haken~\cite{Haken61TheoNormFl}, 
the $3$--sphere recognition problem by Rubinstein and Thompson~\cite{Rubinstein953SphereRec, 
Rubinstein97PolMinSurf, Thompson943SphereRec}, and the 
homeomorphism problem for Haken 3--manifolds by Haken, Hemion and Matveev~\cite{Haken62HomeomProb3Mflds, Hemion, Mat2003}.

A normal surface in a triangulated 3--manifold is decomposed into triangles and quadrilaterals. It is well known that the topology of a normal surface is determined by the quadrilaterals in its cell structure. In this paper, we give a sharp linear bound on the genus $g$ of a closed, orientable normal surface in terms of the number $q$ of quadrilateral discs in the surface; namely $3q\ge 2g.$ This bound is sharp and significantly improves the previously known bound $7q\ge 2g$ due to Kalelkar~\cite{Kalelkar08ChiQuadrNS}. 

All triangulations in this paper are assumed to be semi-simplicial (alias singular) and all manifolds are orientable, unless stated otherwise.
If one restricts the class of triangulations or the class of normal surfaces, one can improve the above bound. 
For instance, in the case of simplicial 3--manifolds satisfying extra hypotheses, the third author~\cite{Spreer10NormSurfsCombSlic} established the bound $q\ge 4g$ for certain normal surfaces. We establish the bound $7q\ge 6g$ for arbitrary normal surfaces in simplicial 3--manifolds or minimally triangulated irreducible 3--manifolds. We also show that for incompressible surfaces in arbitrary triangulated 3--manifolds, one has $q\ge 2g.$ This gives a very simple, new test for compressibility. Moreover, we show that for any oriented normal surface $S$ in an arbitrary triangulated 3--manifold, we have $q\ge ||\, [S]\, ||,$ where the right hand side is the Thurston norm of the homology class represented by $S.$

Our work is combined with bounds by Burton and Ozlen~\cite{burton10-tree} to give bounds on the genus of a vertex normal surface in terms of the number of tetrahedra of the triangulation, and hence on the smallest genus of an incompressible surface in terms of the complexity of a 3--manifold.
The material described up to now, as well as additional special cases (such as 1--sided surfaces or surfaces with boundary) is given in Section~\ref{sec:genus bound} and complemented with an extended set of examples in Section~\ref{sec:Examples}.

We give two applications of our newly obtained bounds.

$\blacktriangleright$ In Section \ref{sec:main results} we use (a corollary to) the bound for 
essential surfaces to characterise the minimal triangulations
of the cartesian product of an orientable surface of genus $g$
and an interval, $F \times I.$ A key feature of any triangulation of this manifold is that it
contains a canonical splitting surface of genus at
least $g$ and having at least $2g$ quadrilaterals.
The resulting lower bound of $10g-4$ on the number of tetrahedra 
of any triangulation of $F \times I$ is attained by the minimal triangulations, which all arise as inflations of the cones of 1--vertex triangulations of $F.$

$\blacktriangleright$ Given a combinatorial orientable surface $S$, the problem of finding 
a polyhedral embedding of $S$ into $\mathbb{R}^3$ is known as the {\em realisation
problem}. One particular sub-problem is
to find realisable surfaces, where the genus is as large as possible
with respect to the number of vertices of the surface.
The state-of-the-art technique to obtain the best known lower bound for the genus is due to Ziegler~\cite{Ziegler08PolSurfsOfHighGenus}, who projects $2$--dimensional subcomplexes
of polytopes into $\mathbb{R}^3.$ However, experimental evidence suggests great potential for improvement of this bound and thus new techniques to tackle this problem are highly sought after. In Section~\ref{sec:polReal}
we show that the dual method (using normal surfaces instead of 
subcomplexes) cannot yield any improved bounds. Given the generality of this approach and the similarity to the
powerful subcomplex method this is a surprising result,
which gives new insights into the well-studied realisation problem.

\textbf{Acknowledgements}\quad The first author is partially supported by the Grayce B.\thinspace Kerr Foundation. The third author is supported by the Australia-India Strategic Research Fund (project number AISRF06660).
The fourth author is partially supported under the Australian Research Council's Discovery funding scheme (project number DP130103694), and thanks the Max Planck Institute for Mathematics, where parts of this work have been carried out, for its hospitality.



\section{Preliminaries}
\label{sec:preliminaries}

The notation and terminology of \cite{JR} and \cite{ti} will be used in this paper, and is briefly recalled in this section. Only the material in \S\ref{subsec:Triangle and quadrilateral regions} is not part of the standard repertoire: here, we define \emph{quadrilateral regions} and \emph{triangle regions} in normal surfaces.


\subsection{Triangulations}

A triangulation $\tri$ consists of a union of pairwise disjoint 3--simplices, $\widetilde{\Delta},$ a set of face pairings, $\Phi,$ and a natural quotient map $p\co \widetilde{\Delta} \to \widetilde{\Delta} / \Phi = M,$ which is required to be injective on the interior of each simplex of each dimension. Here, $\widetilde{\Delta}$ is given the natural simplicial structure with four vertices for each 3--simplex. It is customary to refer to the image of a 3--simplex as a \emph{tetrahedron in $M$} (or \emph{of the triangulation}) and to refer to its faces, edges and vertices with respect to the pre-image. Images of 2--, 1--, and 0--simplices, will be referred to as \emph{faces}, \emph{edges} and \emph{vertices in $M$} (or \emph{of the triangulation}), respectively. The quotient space $M$ is a \emph{pseudo-manifold} (possibly with boundary), and the set of non-manifold points is contained in the 0--skeleton.

The \emph{degree} of an edge in $M$ is the number of 1--simplices in $\widetilde{\Delta}$ that map to it. A triangulation of $M$ is \emph{minimal} if it minimises the number of tetrahedra in $M.$ 

\medskip
If a triangulation $M$ is also a simplicial complex we say that $M$ is a {\em simplicial triangulation}. A simplicial triangulation in which a simplicial neighborhood of each vertex is a simplicial triangulation of the $2$-sphere is referred to as a {\em combinatorial $3$-manifold}. By construction, given an arbitrary triangulation of a closed and compact $3$-manifold, its second barycentric subdivision is a combinatorial $3$-manifold. A combinatorial $3$-sphere is called {\em polytopal} if it is isomorphic to 
the boundary complex of a convex $4$-polytope. Note that not all combinatorial $3$-spheres
are polytopal whereas all simplicial triangulations of the $2$-sphere are isomorphic to the boundary complex of a convex $3$-polytope
\cite{Steinitz06UberEulPolyederRel}.


\subsection{Complexity}
\label{prelim:complexity}

There are different approaches to define the complexity of a 3--manifold. In this paper, the \emph{complexity} $c(M)$ of the compact 3--manifold $M$ is the number of tetrahedra in a minimal (semi-simplicial) triangulation. It follows from the definition that for every integer $k,$ there are at most a finite number of 3--manifolds with  complexity $k$, and it is shown in \cite{JRT-2}, that there is at least one closed, irreducible, orientable 3--manifold of complexity $k.$ Given a closed, irreducible 3--manifold, this complexity agrees with the complexity defined by Matveev~\cite{Mat1990} unless the manifold is $S^3,$ $\R P^3$ or $L(3,1).$ The complexity for an infinite family of closed manifolds has first been given in \cite{JRT}.  The complexity determined here for the infinite family of manifolds with boundary of the form $F\times I$ adds to the known complexities  for handlebodies, where a straight forward Euler characteristic argument gives that $c(\mathbb{H}_g) = 3g-2$, where $\mathbb{H}_g$ is the handlebody of genus $g\ge 1$.

Matveev's complexity of a 3--manifold is defined as the minimal number of true vertices in an almost simple spine for the manifold. It has the following finiteness property: For every integer $k,$ there exist only a finite number of pairwise distinct compact, irreducible, boundary irreducible 3--manifolds that contain no essential annuli and have complexity $k.$ This complexity has been computed for various infinite families of hyperbolic 3--manifolds with one totally geodesic boundary component, see for instance \cite{FV} and \cite{FMP}. However, if one removes the hypothesis on essential annuli, there may be infinitely many 3--manifolds of a given Matveev complexity. In particular, the manifolds of the form $F\times I,$ where $F$ is a closed, orientable surface, and $I$ is a closed interval, have Matveev complexity equal to zero, and we determine their complexity in \S\ref{subset:complexity of FxI}.


\subsection{Normal surfaces}

A {\em normal surface} $S$ in $M$ is a properly embedded
surface that meets each tetrahedron $\Delta$ of $M$ in a disjoint
collection of triangles and quadrilaterals, each running between
distinct edges of $\Delta$, as illustrated in Figure
\ref{fig:normalSubsets}. There are four \emph{triangle types} and three
\emph{quadrilateral types} according to which edges they meet. Within each
tetrahedron there may be several triangles or quadrilaterals of any
given type; collectively these are referred to as {\em normal pieces}.
The intersection of a normal piece of a tetrahedron with one of its
faces is called {\em normal arc}; each face has three \emph{arc types}
according to which two edges of the face an arc meets.

\begin{figure}[htb]
	\centering
	\includegraphics[width=.90\textwidth]{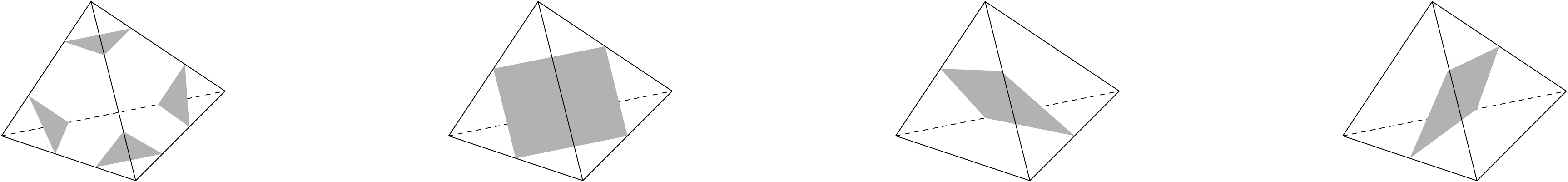}
	\caption{Normal triangles and quadrilaterals within a tetrahedron.}
	\label{fig:normalSubsets}
\end{figure}

Counting the number of pieces of each type for
a normal surface $S$ gives rise to a $7$-tuple per 
tetrahedron of $M$ and hence a $7n$-tuple of non-negative integers describing
$S$ as a point in $\mathbb{R}_{\geq 0}^{7n}$, called its {\em normal coordinates}. 
Such a point must satisfy a set of linear homogeneous
{\em matching equations} (one for each arc type of each internal face).
The solution set to these constraints in $\mathbb{R}_{\geq 0}^{7n}$
is a polyhedral cone whose cross-section polytope
is called the {\em projective solution space}.
Not all of the rational points in this polytope
give rise to a valid normal surface:
For each tetrahedron, at most one of the
three quadrilateral coordinates can be
non-zero. This condition is called the {\em quadrilateral constraints}, and it can
be shown that each rational point in the projective solution space satisfying
the quadrilateral constraints corresponds to a normal surface. These
points and their coordinates are called {\em admissible}.
If a normal surface corresponds to a vertex in the projective solution space
it is called a {\em vertex normal surface}, or {\em extremal surface} 
meaning that its coordinates lie on an extremal ray of the solution cone.
The easiest example of such a vertex normal surface is the boundary of a small
neighborhood around a vertex, called a {\em vertex link}---if $M$ is a manifold, then this is necessarily
a sphere or disc consisting entirely of triangles.

\medskip
Due to work by Tollefson \cite{Tollefson98QuadTheory} 
we know that any normal surface without vertex linking
components is determined by its quadrilaterals and hence by a vector in
$\mathbb{R}_{\geq 0}^{3n}$. In this case, the matching equations are given
by the intersection of the quadrilaterals and the edges of the triangles
called the {\em $Q$-matching equations}. Intuitively, these equations arise from the fact that as one
circumnavigates the earth, one crosses the equator from north to south
as often as one crosses it from south to north.
We now give the precise form of these equations.
To simplify the discussion,
we assume that $M$ is oriented and all tetrahedra are given
the induced orientation; see \cite[Section~2.9]{ti} for
details.

\begin{figure}[h]
    \centering
    \subfigure[Abstract neighbourhood $B(e)$]{%
        \label{fig:matchingquadbdry}%
        \includegraphics[scale=1.1]{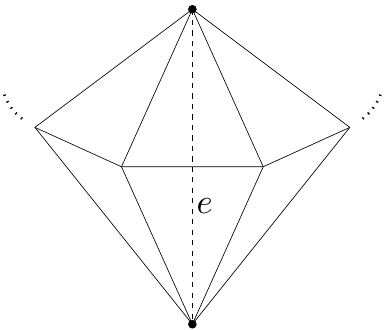}}
    \qquad
    \subfigure[Positive slope $+1$]{%
        \label{fig:matchingquadpos}%
        \includegraphics[scale=1.1]{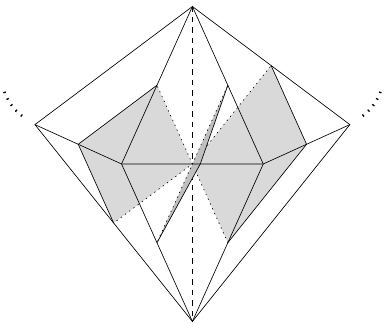}}
    \qquad
    \subfigure[Negative slope $-1$]{%
        \label{fig:matchingquadneg}%
        \includegraphics[scale=1.1]{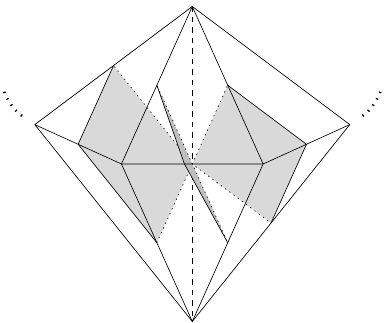}}
    \caption{Slopes of quadrilaterals}
    \label{fig:slopes}
\end{figure}

Consider the collection $\mathcal{C}$ of all (ideal) tetrahedra meeting
at an edge $e$ in $M$ (including $k$ copies of tetrahedron $\sigma$ if $e$ occurs
$k$ times as an edge in $\sigma$).  We form the \emph{abstract
neighbourhood $B(e)$} of $e$ by pairwise identifying faces of tetrahedra
in $\mathcal{C}$ such that there is a well defined quotient map from
$B(e)$ to the neighbourhood of $e$ in $M$; see
Figure~\ref{fig:matchingquadbdry} for an illustration.  Then  $B(e)$ is a ball
(possibly with finitely many points missing on its boundary).  We think
of the (ideal) endpoints of $e$ as the poles of its boundary sphere, and
the remaining points as positioned on the equator.

Let $\sigma$ be a
tetrahedron in $\mathcal{C}$. The boundary square of a normal
quadrilateral of type $q$ in $\sigma$ meets the equator of $\partial
B(e)$ if and only it has a vertex on $e$. In this case, it has a slope
$\pm1$ of a well--defined sign on $\partial B(e)$ which is independent of the
orientation of $e$. Refer to Figures~\ref{fig:matchingquadpos} and
\ref{fig:matchingquadneg}, which show quadrilaterals with
\emph{positive} and \emph{negative slopes} respectively.

Given a quadrilateral type $q$ and an edge $e,$ there is a
\emph{total weight} $\wt_e(q)$ of $q$ at $e,$ which records the sum of
all slopes of $q$ at $e$ (we sum because $q$ might meet $e$ more than
once, if $e$ appears as multiple edges of the same tetrahedron).
If $q$ has no corner on $e,$ then we set
$\wt_e(q)=0.$ For normal coordinates $x$, the set of quadrilateral
types $\square$, and an edge $e$ in $M,$ the $Q$--matching equation of $e$
is then defined by $0 = \sum_{q\in \square}\; \wt_e(q)\;x(q)$.

The following result is related to Haken's Hauptsatz 2 in \cite{Haken61TheoNormFl}, see \cite[Theorem~2.4]{ti} for a proof in the setting of this paper.

\begin{theorem}\label{thm:admissible integer solution gives normal}
For each $x\in \R^{3n}$ with the properties that $x$ has integral coordinates, $x$ is admissible and
$x$ satisfies the $Q$--matching equations, there is a (possibly
non-compact) normal surface $S$ such that $x = x(S).$ Moreover, $S$ is
unique up to normal isotopy and adding or removing vertex linking surfaces,
i.e., normal surfaces consisting entirely of normal triangles.
\end{theorem}


\subsection{Triangle and quadrilateral regions}
\label{subsec:Triangle and quadrilateral regions}

Let $S$ be a normal surface in a triangulated, compact 3--manifold $M.$ Denote by $S_\Delta$ the subcomplex of $S$ made up of all triangles in $S,$ $S_\square$ the subcomplex made up of all quadrilaterals, and $S^{(0)}$ the set of its vertices. A \emph{triangle region} in $S$ is the closure in $S$ of a connected component of $S_\Delta \setminus S^{(0)},$ and a \emph{quadrilateral region} in $S$ is the closure in $S$ of a connected component of $S_\square \setminus S^{(0)}.$ (Kalelkar~\cite{Kalelkar08ChiQuadrNS} calls these regions \emph{strongly connected}.)

All triangles in a triangle region link the same vertex of $M.$ In fact, since $S$ contains at most finitely many normal triangles,
a triangle region in $S$ never contains two normally isotopic triangles. 
To see this, suppose $\Delta_a$ and $\Delta_b$ are two
normally isotopic triangles in same triangle region $R.$ 
Then $\Delta_a$ and $\Delta_b$ are contained in some tetrahedron $\sigma_0.$ 
Without loss of generality, we may assume that the corner cut off by $\Delta_a$ from $\sigma_0$ contains no normal discs, and the corner cut off by $\Delta_b$ contains only $\Delta_a.$ Since $R$ is a triangle region, there is a path of normal triangles 
$$ \Delta_b = \Delta_0, \,\, \Delta_1,\,\, \ldots \,\, , \,\,\Delta_k = \Delta_a$$
with the property that subsequent triangles share a common edge, and we may assume that all these normal triangles are pairwise distinct. We denote by $\sigma_j$ the tetrahedron containing $\Delta_j,$ and note that these tetrahedra are not necessarily pairwise distinct.
Since $\Delta_1$ is glued to $\Delta_b$ in $R$ along a normal arc, it follows that $\Delta_a$ is glued in $S$ to a normal disc contained in the corner cut off by $\Delta_1$ in $\sigma_1.$ Thus $\Delta_a$ is glued in $S$ to a normal triangle, denoted $\Delta_1',$ along a normal arc. It follows that $\Delta_1'$ is also contained in $R.$ Iterating this argument gives, for each $j,$ a normal triangle $\Delta_j'$ in $R$ which is in the corner cut off by $\Delta_j$ in $\sigma_j.$ But this contradicts the fact that the corner in $\sigma_k$ cut off by $\Delta_k = \Delta_a$ contains no normal triangle.

A triangle region is a 2--complex, and not necessarily a surface with boundary. We therefore say that 
the \emph{topological type} of a triangle region $R$ in $S$ is the topological type of its interior. From the above discussion, we know that this is the topological type of a topologically finite planar surface. Given a triangle region $R$ in $S,$ choose a compact core $C$ of the planar surface ${int}(R).$ Then $C$ is a compact planar surface, and each of its boundary components naturally corresponds to a graph made up of quadrilateral edges. These graphs are termed \emph{chains of (unglued) quadrilateral edges}. 




\section{Bounds on genera of normal surfaces}
\label{sec:genus bound}

Given the closed, orientable, connected surface $S$ of genus $g=g(S),$ there are at least $2g$ branches in any spine for $S.$
If $S$ is a normal surface in the triangulated 3--manifold $M,$ denote $\Quad(S)$ the number of quadrilateral discs in $S.$ We will give several bounds on the genus of $S,$ and indicate whether or not they are known to be sharp.


\subsection{Quadrilateral surfaces}

The first bound is based on a simple Euler characteristic argument and applies to surfaces entirely made up of quadrilaterals.

\begin{lemma}[Bound for quadrilateral surfaces]\label{lem:inequ for quad surf}
Suppose $M$ is a triangulated, compact, orientable 3--manifold, and let $S$ be a closed, connected, orientable normal surface in $M$ consisting entirely of quadrilaterals and having exactly $v$ vertices in its cell structure. Then $q(S) = 2 g(S) + v-2.$ In particular, if $v\ge 2,$ then $q(S) \ge 2 g(S).$ If $v=1,$ then for each quadrilateral in $S$, there is a vertex normal surface in $M$ having exactly one quadrilateral disc (and possibly some triangles) in its cell structure.
\end{lemma}

\qquad $\blacktriangleright$ Examples~\ref{ex:quadEx1} and \ref{ex:quadEx2} show that this bound is sharp.\footnote{All examples are collected in \S\ref{sec:Examples}. Brief remarks are indicated with $\blacktriangleright$ as in the text here.}\\

\begin{proof}
The proof is a simple Euler characteristic argument. Note that for a quadrilateral surface, the number of edges is exactly twice the number of quadrilaterals. Thus the Euler characteristic formula gives us $\chi(S) = v - 2q(S) + q(S) = v - q(S)$. Substituting $\chi(S) = 2 - 2 g(S)$ yields $q(S) = 2 g(S) + (v-2).$ 

Now suppose $S$ has exactly one vertex in its cell structure. In this case, all corners of quadrilaterals are identified and $S$ meets each tetrahedron in at most one quadrilateral disc. Since all corners of quadrilaterals are identified and $[S]_{Std} = \sum Q_k,$ where $Q_i\neq Q_j,$ we see that each $Q_k$ is a solution to the $Q$--matching equations. In particular $Q_k$ is a vertex solution to the $Q$--matching equations.
\end{proof}


\subsection{Closed normal surfaces and applications}

\begin{theorem}[Bound for closed normal surfaces]\label{thm:genus}
	Let $M$ be a triangulated, compact, orientable 3--manifold, 
	and $S$ be a closed, connected, orientable normal surface in $M.$ Then
	$$3 \Quad(S)\ge 2 g(S).$$
\end{theorem}

\qquad $\blacktriangleright$ This improves the bound of $7 \Quad(S) \ge 2 g(S)$ given by Kalelkar~\cite{Kalelkar08ChiQuadrNS}.\\
\qquad $\blacktriangleright$ Examples~\ref{ex:genleq3} and \ref{ex:nonor} show that this is a sharp bound.\\

We will give two proofs: the first is short; the second (given in the next subsection) provides extra insight in the structure of orientable normal surfaces, that will be used in some of the corollaries.

\begin{proof}[First proof of Theorem~\ref{thm:genus}] First note that the inequality clearly holds if $S$ is a sphere; it also holds if $S$ is a torus since a closed normal surface with no quadrilateral discs is a vertex linking sphere.

The triangle regions in $S$ are planar surfaces (see \S\ref{subsec:Triangle and quadrilateral regions}). We will construct a new cell structure on $S$ by modifying the original one as follows: Let $R$ be a triangle region. Given an edge incident with two triangles in $R$ and connecting different boundary components of $R,$ we will shrink this edge to a point, turning the two adjacent triangles into bigons (we term this process \emph{collapsing}). This connects the two boundary components. We will then continue in this fashion until all the boundary components of $R$ have been joined into a single connected graph.

If a normal triangle in $R$ connects three distinct boundary components of $R,$ then we will need to shrink at most two of its edges, collapsing the triangle at most into a monogon. Since this is the most we would ever need to collapse a triangle, each triangle collapses to either a triangle, bigon or monogon (but not a single vertex). Because the original triangle region was planar, its image after this operation will have simply connected interior. 

Apply the above construction to all triangle regions in $S,$ giving $S$ a new cell decomposition. Then we can find a spine for $S$ carried by edges in this cell structure that are disjoint from the interior of any of the regions corresponding to triangle regions in the original cell decomposition. Such a spine will consist of edges of quadrilaterals from the original cell structure.

If some quadrilateral has all four edges in the spine, then the complement of the spine consists entirely of the interior of the quadrilateral disc. But then the four edges of the quadrilateral must be identified in pairs, and hence $S$ is a sphere or a torus. We already know that the inequality holds in this case. So we may assume that each quadrilateral has at most three edges in the spine, and so there are at most $3 \Quad(S)$ branches in the spine. Since the number of branches is at least $2 g(S),$ we have the claimed inequality.
\end{proof}

The following corollaries are consequences of the relationship between genus and Euler characteristic. Recall that for a \emph{non-orientable} surface $S,$ one has $\chi(S) = 2 - g(S),$ where the genus $g(S)$ of a non-orientable surface is the number of cross-caps.

\begin{corollary}[Bound for non-orientable normal surfaces]
	\label{for:nonorSurfs}
	Let $M$ be a triangulated, compact, orientable 3--manifold, 
	and $S$ be a closed, connected non-orientable normal surface in $M.$
Then	 $$g(S)  \le 3 \Quad(S)+1.$$
\end{corollary}

\qquad $\blacktriangleright$ Example~\ref{ex:nonor} shows that this bound is sharp with $q=1$ and $g=4.$


\begin{proof}
Since $M$ is orientable, doubling the normal surface coordinate of a non-orientable normal surface gives the coordinate of the orientable double cover of the surface. The (orientable) genus of the double cover will be equal to the (non-orientable) genus of the original surface, but the number of quadrilaterals will be double. Thus Theorem~\ref{thm:genus} implies the inequality.
\end{proof}

\begin{corollary}[Bound for normal surfaces with boundary]
	\label{cor:boundedSurfs}
	Let $M$ be a triangulated, compact, orientable 3--manifold, 
	and $S$ be a closed, connected, orientable normal surface in $M$ with $b>0$ boundary components.
Then	 $$ 2 g(S) +b \le 3 \Quad(S) + 1.$$
\end{corollary}

\qquad $\blacktriangleright$ Example~\ref{ex:bounded} shows that this bound is sharp
	with $g=q=1$, and $b = 2$.

\begin{remark}
	Note that the number of boundary components $b$ does not need to be known in order to check 
	if equality in Corollary~\ref{cor:boundedSurfs} is satisfied: The genus of a bounded surface $S$ with $b$
	boundary components is given by $g(S) = 1 - (\chi (S) + b)/2 $.
	Hence, $2g(S) + b = 2 - \chi(S) - b + b$ and we have
	$$ 2 - \chi(S) \le 3 \Quad(S) + 1.$$
\end{remark}

\begin{proof}
Double $M$ along its boundary, notice that the double of $S$ is a normal surface in the induced triangulation with twice as many quadrilateral discs, and apply Theorem~\ref{thm:genus}.
\end{proof}

\begin{corollary}[Bound for Haken sum]
\label{cor:genus Haken sum}
	Let $M$ be a triangulated, compact, orientable 3--manifold, and $S$ be 
	a closed, connected, orientable normal surface in $M.$ Suppose that the disjoint union of $S$ and $m$ vertex linking spheres 
	is the Haken sum of $n$ closed, connected, orientable normal surfaces. Then 
	$$ 2 g(S)  \le 3 \Quad(S)+2(1-n+m) .$$
	\end{corollary}

\qquad $\blacktriangleright$ Example~\ref{ex:hakenSum} shows that this bound is sharp for non-trivial Haken sums.

\begin{proof}
By hypothesis, we have a Haken sum of the form $S+ \sum m_jV_j = \sum n_kF_k,$ where $\sum m_j = m$ and $\sum n_k = n.$ Linearity of Euler characteristic for Haken sums and Theorem~\ref{thm:genus} applied to each $F_k$ yields the result.
\end{proof}


\subsection{A second proof and its applications}
\label{subsect:second proof}

\begin{proof}[Second proof of Theorem~\ref{thm:genus}]
Suppose $M$ is a triangulated, compact, orientable 3--manifold, and $S$ is a closed, connected, orientable normal surface in $M$. We give each $M$ and $S$ an orientation. This determines a \emph{transverse orientation} of $S,$ and hence of all normal discs and arcs in $S$ (see Figure~\ref{fig:TONS}) as well as their pre-images in $\widetilde{\Delta}.$ Before we use this extra structure to study $S,$ we recall the following notions from \cite{TONS}.

	\begin{figure}[h]
		\begin{center}
			\includegraphics[width=\textwidth]{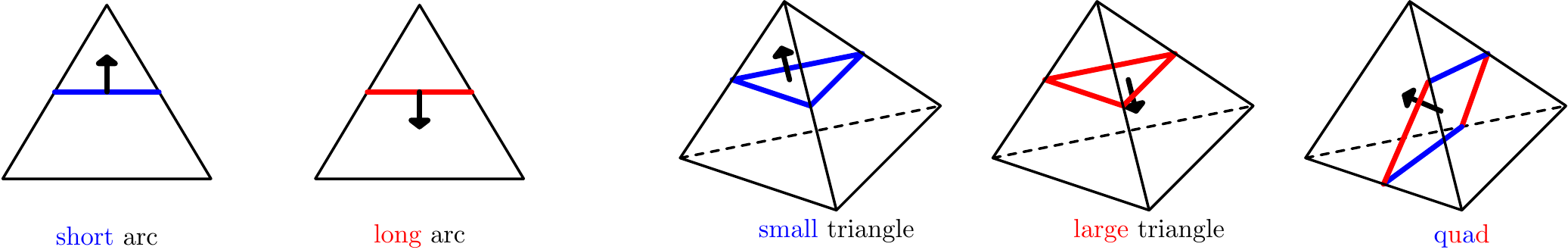}
		\end{center}
		\caption{Transversely oriented normal arcs and discs. \label{fig:TONS}} 
	\end{figure}

Let $\sigma$ be a 2--simplex and $\alpha \subset\sigma$ be a transversely oriented normal arc. The transverse orientation can be viewed as a function, which maps one component of $\sigma \setminus \alpha$ to $+1$ and the other to $-1.$ We say that the maximal subcomplex of $\sigma$ contained in the component of $\sigma \setminus \alpha$ having positive sign is \emph{dual} to $\alpha.$ This subcomplex is either a 0--simplex, in which case we call $\alpha$ a \emph{short arc}; or a 1--simplex, in which case $\alpha$ is a \emph{long arc}. See Figure~\ref{fig:TONS}.

The transverse orientation on each normal disc in a 3--simplex induces a transverse orientation of the normal arcs in its boundary. Each triangle will either have all three edges dual to the same vertex or all three edges dual to different edges. In the first case, we will say that the triangles are \emph{small}. In the second case, we will say that they are \emph{large}. Two opposite edges of each quadrilateral will be long, and will be dual to the same edge of the 3--simplex. We will say that the quadrilateral is \emph{dual to} this edge. The other two edges will be short. 

The notions of short and long edges descend from $\widetilde{\Delta}$ to the triangulation of $M$ since they are preserved by the face pairings. In particular, 
the definition of short and long edges is not relative to the polygon that it is contained in. So if two quadrilaterals have an edge in common, then this is either a short edge of both quadrilaterals, or a long edge of both quadrilaterals. (Note that the notions of short/long and small/large are interchanged by changing the orientation of $S.$ Using the long edges for the construction is motivated by the applications.)

We will use these properties of quadrilateral edges to define an equivalence relation on the set of all long edges in the quadrilateral subcomplex of $S$. Call the long edges $e$ and $f$ of the quadrilaterals $P$ and $Q$ respectively equivalent if there is a chain of quadrilaterals $P = Q_0, \ldots, Q_k = Q$ with the property that successive quadrilaterals are identified along long edges. In particular, the two long edges of one quadrilateral are equivalent.

	\begin{figure}[h]
		\begin{center}
			\includegraphics[width=.4\textwidth]{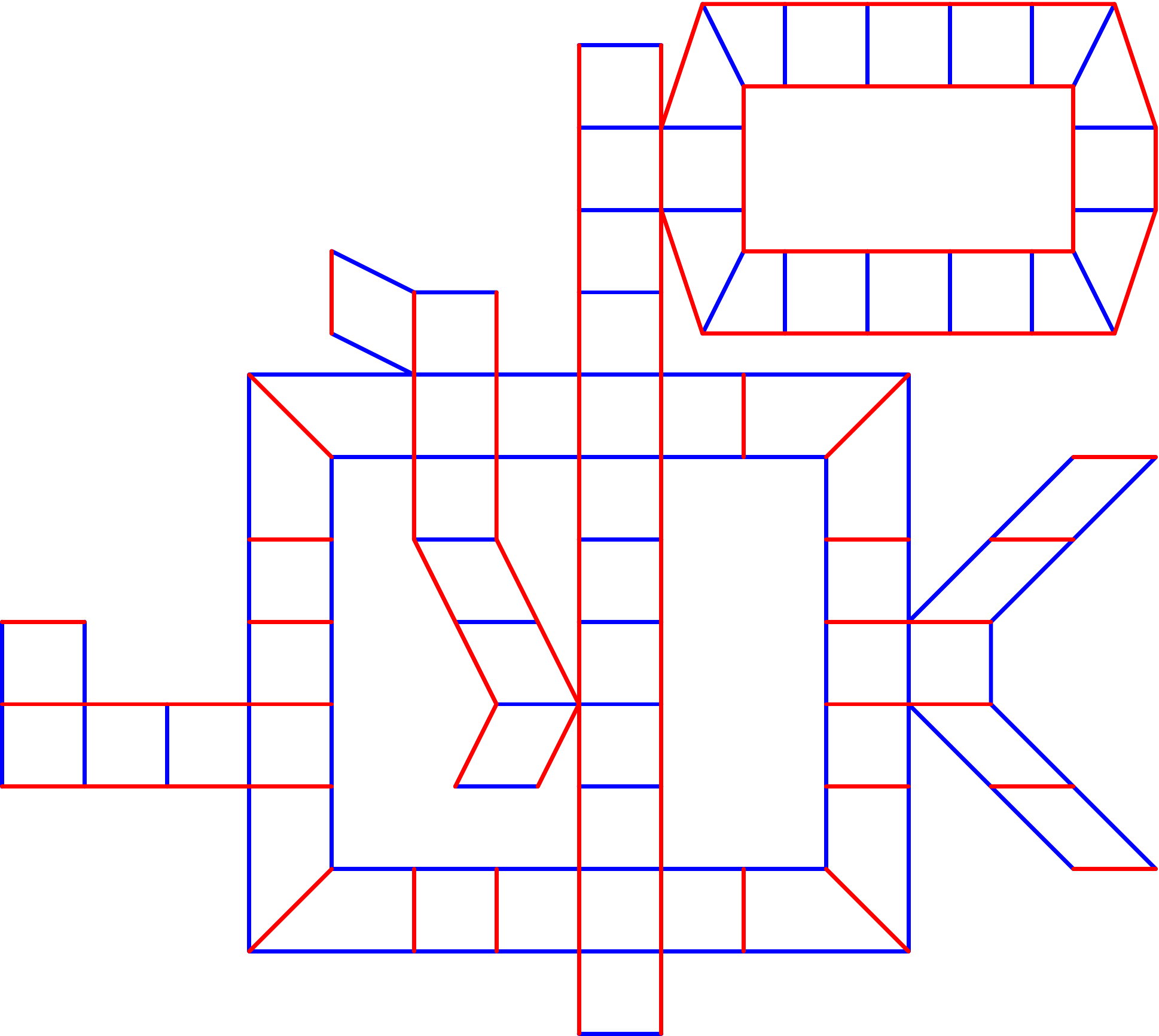}
		\end{center}
		\caption{Equivalence classes of long quadrilaterals and vertical short edge paths. \label{fig:long and short edge classes}} 
	\end{figure}

We will again define a new cell structure on $S.$ As above, we pinch the boundary components of each non-simply connected triangle region together by shrinking edges in the region that connect disconnected boundary components of each triangle region. As before, the interior of each resulting region will be simply connected. Denote the resulting surface $\widetilde{S},$ and note that it is homeomorphic to $S,$ and that a spine of $\widetilde{S}$ is again contained in the union of all quadrilateral edges. 

But we can restrict the locus of the spine even further: Consider the graph $\Gamma$ on $\widetilde{S},$ consisting of the union of all short edges and precisely one long edge from each equivalence class of long edges. We claim that the complement of $\Gamma$ in $\widetilde{S}$ is a union of pairwise disjoint open discs. Consider a vertical path of quadrilateral discs. If this closes up on itself (in which case the quadrilaterals form an annulus linking a common edge in $M$), then our construction yields an open disc formed by cutting the open annulus along one of the long edges. If this does not close up on itself, then the extremal quadrilaterals connect to pinched simply-connected triangle regions (possibly the same). Cutting the union of this simply-connected region with the chain of quadrilaterals along any long edge in the chain results in one or two open discs, and hence again results in simply connected regions. One can now iterate this procedure over all vertical paths of quadrilateral discs.

Since the complement in $\widetilde{S}$ of $\Gamma$ is a collection of (open) discs, a spine for $\widetilde{S}$ can be chosen in $\Gamma.$ Since each quadrilateral meets $\Gamma$ in at most two short edges and at most one long edge, the spine has at most $2q+q=3q$ edges.
\end{proof}

\begin{corollary}[Bound for incompressible normal surfaces]\label{cor:genus of essential}
Let $M$ be a triangulated, compact, orientable 3--manifold,  and $S$ be a closed, connected, orientable normal surface in $M.$ If $S$ is incompressible, then $$2 g(S)\le \Quad(S).$$
\end{corollary}

\qquad $\blacktriangleright$ This bound is sharp for all $g \geq 1$ for manifolds with boundary (see \S\ref{subsec:minimal triangulations of FxI} and Example~\ref{ex:incompressible}).

\qquad $\blacktriangleright$ An incompressible torus with 2 quads in a closed 3--manifold  is also given in 
 Example~\ref{ex:incompressible}.

\qquad $\blacktriangleright$ The contrapositive certifies compressibility of many of our examples in \S\ref{sec:Examples}.

\begin{proof}
Suppose $S$ is an incompressible, closed, connected, orientable normal surface in $M.$ We may assume that $g(S)\ge 1.$ In particular, $S$ contains at least one quadrilateral disc. We will modify the second proof of Theorem~\ref{thm:genus} and make some additional observations.

First suppose that there is a triangle region, say $R,$ in $S,$ which is not simply connected. Since every simple closed loop in $\partial R$ bounds a disc in $M,$ it also bounds a disc on $S.$ We may choose a simple closed curve $b \subset \partial R$ with the property that the closed disc $D\subset S$ with $\partial D = b$ contains $R.$ In particular, the graph $\Gamma$ can be chosen such that $\Gamma \cap D \subset b;$
the complement of $D$ contains quadrilateral discs; and the chain of quadrilateral edges corresponding to $b$ can be contracted to a point in $S$ (though we will not do this at this stage).

Recall that the long edges $e$ and $f$ of the quadrilaterals $P$ and $Q$ respectively are equivalent if there is a chain of quadrilaterals $P = Q_0, \ldots, Q_k = Q$ with the property that successive quadrilaterals are identified along long edges. The chain of quadrilaterals $P = Q_0, \ldots, Q_k = Q$ identifies successive short edges, and we will term a maximal chain of such short edges a \emph{vertical short edge path}. 

\begin{figure}[h]
		\begin{center}
			\includegraphics[width=.6\textwidth]{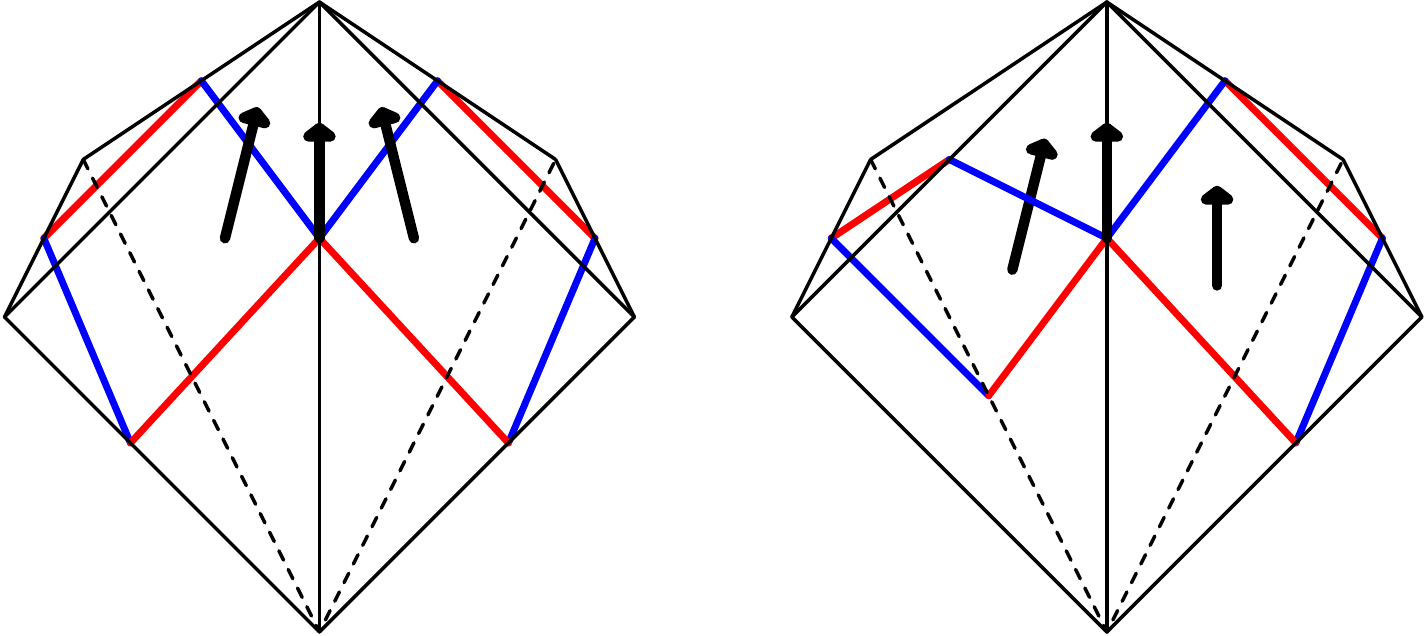}
		\end{center}
		\caption{Loop of short quadrilateral edges can be homotoped into a vertex link \label{fig:homotope short edge loop}} 
	\end{figure}

We now claim that the vertical short edge paths in $\Gamma$ can, one by one, be contracted to points, hence showing that the spine contained in $\Gamma$ arises from at most $\Quad(S)$ long edges. In the iterative process, we still maintain the terms \emph{short quadrilateral edges}, \emph{long quadrilateral edges} and \emph{vertical short edge path} for the images of these objects, even though after a contraction, some quadrilaterals have turned into triangles or bigons, and we denote the surface resulting after $k$ iterations by $S_k$. Suppose that a vertical short edge path in $S_k$ contains a loop $\gamma_k.$ We may assume without loss of generality that $\gamma_k$ is simple. Then the original surface $S$ contains a loop $\tilde{\gamma}_k$ made up of short quadrilateral edges and that maps to $\gamma_k,$ since we have only pinched short edges in the boundary of quadrilateral discs. This loop is normally homotopic into a vertex link (see Figure~\ref{fig:homotope short edge loop}) and hence bounds a disc on $S.$ This disc maps to a disc on $S_k$ with boundary $\gamma_k,$ and so $\gamma_k$ can be contracted to a point. It follows that after all vertical short edge path have been iteratively collapsed, we have a surface homeomorphic to $S$ and with a spine made up of the images of at most $\Quad(S)$ long edges.
\end{proof}

The argument in the above proof can be applied more generally, but pinching a vertical short edge path may then result in compressions of the surface. One can still obtain useful bounds if one has conditions that ensure that the genus of any compression is bounded from below; we illustrate this in two situations.

\begin{corollary}[Bound for the splitting surface of a product]\label{cor:genus of boundary homologous}
Let $M = F \times I,$ where $F$ is a closed, connected, orientable surface, with a triangulation. Suppose $S$ is a closed, connected, orientable normal surface in $M,$ which separates the two boundary components of $M.$ Then $$2 g(F)\le \Quad(S).$$
\end{corollary}

\qquad $\blacktriangleright$ The canonical splitting surfaces in the minimal
triangulations of $F \times I$ given in 
Section~\ref{subsec:minimal triangulations of FxI} show that this bound
is sharp for all values of $g \geq 1$.

\begin{proof}
The argument of the previous proof only needed the fact that a spine for the surface can be chosen outside of regions on the surface which are bounded by curves made up of short quadrilateral edges. For any oriented normal surface $S$ in $M,$ a simple closed curve made up of short quadrilateral edges in $S$ is homotopic into a vertex link and hence bounds a disc $D$ in $M$ with boundary on $S.$ If $D$ also bounds a disc on $S,$ then a spine for $S$ can be chosen disjoint from $D.$ Otherwise, $D$ is a compression disc, and $S\setminus \partial D$  contains a spine for each component of the surface obtained by compressing $S$ along $D.$ Also note that any further compression disc for a component $S'$ of the surface arising from the compression can be chosen disjoint from the disc on $S'$ parallel to $D.$ 

If $M = F \times I$ and $S$ is a surface separating the two boundary components of $M,$ it follows that if one compresses $S$ along \emph{any} sequence of compression discs, the resulting surface has a component that is incompressible, separates the boundary components of $M,$ and has genus at least the genus of $F.$ Let $\Gamma_0$ be a spine for $S_0=S$ contained in the union of all short edges and one edge from each equivalence class of long edges. In the previous proof, vertical short edge paths in $\Gamma_0$ were, one by one, contracted to points. This is now adjusted as follows. Let $\gamma_0$ be a vertical short edge path in $\Gamma_0.$ If this bounds a disc on $S,$ then it is contracted to a point, giving a surface $S_1$ and we denote $\Gamma_1$ the image of $\Gamma_0.$ Otherwise, we may assume that a simple closed loop in the vertical short edge path $\gamma_0$ is the boundary of a compression disc for $S_0.$ We then cut $S_0$ along this loop to obtain a surface with two boundary components, and contract each of the boundary components to a point. If this process results in a connected surface, we denote it $S_1.$ Otherwise we denote $S_1$ a component that separates the two boundary components of $M,$ and $\Gamma_1$ the image of $\Gamma_0$ in this component. We can now, as before, iterate this procedure using the induced cell decomposition. The final surface will have no vertical edge paths left, and hence the final spine $\Gamma$ will consist of at most one edge from each equivalence class of long edges and the genus is bounded below by the genus of $F$. Whence $2 g(F) \le |\Gamma| \le \Quad(S).$
\end{proof}

\begin{corollary}[Bound for Thurston norm]\label{cor:bound for Thurston norm}
Let $M$ be a triangulated, compact, orientable 3--manifold,  and $S$ be a closed, oriented normal surface in $M.$ Then $$||\, [S]\, || \le \Quad(S),$$ where the left hand side is the Thurston norm of the homology class represented by $S.$
\end{corollary}

\begin{proof}
Assuming that $S$ is connected, we make the following adjustment to the previous proof. Instead of discarding components, we keep each component after each compression and terminate when in each component all vertical edge paths have been collapsed. We then delete all components that are spheres or tori. If the resulting surface is empty, then  $||\, [S]\, ||=0$ and there is nothing to prove. Otherwise, denote $S_k$ the resulting oriented surfaces and $\Gamma_k$ their spines consisting of long edges. Then
$$\Quad(S) \ge \sum |\Gamma_k| \ge  \sum 2g(S_k) \ge \sum (2- \chi(S_k)) \ge \sum  - \chi(S_k) \ge ||\,[\cup S_k]\,||=||\,[S]\,||.$$
This completes the proof for the case where $S$ is connected. If $S$ is not connected, we obtain the result by summing the first inequality over all components---the remaining inequalities then apply as above.
\end{proof}

\begin{corollary}[Bound in terms of quads and chains]\label{cor:quad and chain bound}
Let $M$ be a triangulated, compact, orientable 3--manifold, and $S$ a closed, connected, orientable, normal surface in $M.$ If every chain of quadrilateral edges in the boundary of a non-simply connected triangle region in $S$ contains at least $n$ edges, then 
$$2 g(S) \le \Bigg(1 + \frac{4}{n}\Bigg) \Quad(S),$$
where we set $n=\infty$ if there is no non-simply connected triangle region.
\end{corollary}

\qquad $\blacktriangleright$ The incompressible surfaces in Example~\ref{ex:incompressible} show that this bound is sharp for $n=\infty$.

\begin{remark}
The corollary can also be applied with $n$ denoting the minimal number of edges in \emph{any} chain of quadrilateral edges. Note that one needs at least $n=3$ to obtain an improvement on the general bound, and that the above bound is not sharp, as the number of non-simply connected triangle regions has not been taken into account.
\end{remark}

\begin{proof}
We modify the previous proofs as follows. We only pinch the small triangle regions to give simply connected components. In the large triangle regions, we add edges connecting the boundary components of a non-simply connected region. For each non-simply connected region, this adds one less than the total number of boundary components, and we term these edges \emph{long cut edges}. A spine for the surface is then contained in the union of all short quadrilateral edges, one edge from each equivalence class of long quadrilateral edges and the long cut edges. We would again like to pinch the portion of every vertical short edge path, which is contained in the spine, to a point. We can do this successively. If a short edge has two distinct end-points, it can be pinched to a point. If it has identical end-points then there is a loop (on $S$) of short edges. This must correspond to a boundary component of a non-simply connected small triangle region since the initial pinching of small triangle regions only identifies corners of quadrilaterals contained on distinct chains of short quadrilateral edges. This shows that all short edges in the spine can be contracted to points except for at most as many as there are boundary components of small triangle regions. Whence the spine can be chosen to have at most $q + |\partial|$ edges, where  $|\partial|$ is the total number of boundary components of non-simply connected triangle regions. By hypothesis, $n  |\partial| \le 4q,$ giving the desired inequality.
\end{proof}

\begin{corollary}[Bound for normal surfaces in simplicial manifolds]\label{cor:simplicial}
Let $M$ be a triangulated, compact, orientable 3--manifold, and $S$ be a closed, connected, orientable normal surface in $M.$ If 
the triangulation of $M$ is simplicial, then
$$6 g(S) \le 7 \Quad(S).$$
\end{corollary}

\begin{proof}
We may apply Corollary~\ref{cor:quad and chain bound} with $n= 3.$
\end{proof}

\begin{remark}
The bound in Corollary~\ref{cor:simplicial} is not known to be sharp. Examples of normal surfaces $S$ can 
be constructed with $q(S) = 3g(S) + O(\sqrt{g(S)})$ quadrilaterals (see Example~\ref{ex:simplicial} for a discussion of these examples).
All these examples are well within the conjectured bound of $q(S) \geq 3g(S)$ from \cite{Spreer10NormSurfsCombSlic}.
\end{remark}

\begin{corollary}[Bound for normal surface in minimal, prime manifold]\label{cor:minimal}
Let $M$ be a triangulated, compact, orientable, prime 3--manifold, and $S$ be a closed, connected, orientable normal surface in $M.$ If the triangulation of $M$ is minimal, then
$$6 g(S) \le 7 \Quad(S).$$
\end{corollary}


\begin{proof}
The proof is divided into two cases. If the triangulation consists of one or two tetrahedra, then one can verify the conclusion, for instance, using Regina~\cite{Regina}, for all fundamental surfaces in the finite list of prime manifolds of complexity up to two. The case of a general connected surface in these manifolds then follows as in the proof of Corollary~\ref{cor:genus Haken sum}. 

Hence assume that there are at least 3 tetrahedra in the triangulation. We will show that work by Jaco-Rubinstein~\cite{JR} and Burton~\cite{BAB-2004, BAB-2007} allows us to apply Corollary~\ref{cor:quad and chain bound} with $n= 3.$ Since $M$ is prime and the minimal triangulation has at least 3 tetrahedra, it is $0$--efficient (see \cite{JR}). If a quadrilateral edge in the orientable surface $S$ forms a loop, then some face is a cone. Corollary~5.4 in \cite{JR} now implies $M=S^3,$ contradicting the fact that minimal triangulations of $S^3$ have one tetrahedron. If two short quadrilateral edges in $S$ form a bigon, then two faces in the triangulation form a cone (possibly with further self-identifications), and the triangulation is again not minimal due to \cite[Lemma~2.7 and Corollary~2.10]{BAB-2004}  and \cite[Lemma~3.6 and Corollary~3.8]{BAB-2007}.
\end{proof}

\begin{remark}
The same bound $6 g(S) \le 7 \Quad(S)$ applies to the face-generic, face-pair reduced triangulations of Luo-Tillmann~\cite{LT2013}.
\end{remark}


\subsection{Bounds in terms of the size of the triangulation}

Improving upon bounds of Hass, Lagarias and Pippenger~\cite{HLP} for vertex normal surfaces in simplicial triangulations, Burton and Ozlen~\cite{burton10-tree} showed that the maximal coordinate of a vertex normal surface in a semi-simplicial
triangulation of an orientable closed $3$-manifold is at most $(4n^2 + 2) (\sqrt{6})^n$, where
$n$ is the number of tetrahedra. The quadrilateral constraints imply that no vertex normal
surface in a closed orientable $3$-manifold triangulation can have more than $n (4n^2 + 2) (\sqrt{6})^n$
quadrilaterals. Theorem~\ref{thm:genus} and Corollary~\ref{cor:genus of essential} therefore have the following consequences.

\begin{corollary}
	Let $M$ be a triangulated, compact, orientable 3-manifold. Suppose the triangulation has $n$ tetrahedra and
	$S$ is a closed, orientable vertex normal surface $S$ in $M.$ Then
	\begin{equation}
		\label{eq:normal}
		g(S) \le (6n^2 + 3) (\sqrt{6})^n.
	\end{equation}
	If, in addition, $S$ is incompressible, then
	\begin{equation}
		\label{eq:incompressible}
		g(S) \le (2n^2 + 1) (\sqrt{6})^n.
	\end{equation}
\end{corollary}

\begin{remark}
	Equation~(\ref{eq:normal}) also follows from an elementary counting argument:
	Let $S$ be an orientable vertex normal surface in $M$ with $v(S)$ vertices, then using the bounds on triangle and quadrilateral coordinates from \cite{burton10-tree} one obtains
	$$ g(S) \le (6n^2 + 3) (\sqrt{6})^n + \frac{2-v(S)}{2}. $$
	This equation can be improved further by giving a lower bound
	on $v(S)$. In contrast, equation~(\ref{eq:incompressible}) cannot be derived from \cite{burton10-tree} and combined with  \cite{JO} has the following immediate application.
\end{remark}

\begin{corollary}
	Let $M$ be a compact, orientable 3-manifold with complexity $c = c(M).$ Then the minimal genus $g$ of an incompressible, closed, orientable surface in $M$ satisfies 
$$
		g \le (2c^2 + 1) (\sqrt{6})^c.
$$
\end{corollary}

\medskip
In \cite{Burton12CompTopNormSurfExperiments} there are examples of families of triangulations
containing normal surfaces with exponentially large normal coordinates. However, these normal
surfaces are discs and spheres.

%



\section{Minimal triangulations of $\mathbf{F}\times \mathbf{I}$}
\label{sec:examples}

We now determine the complexity and all minimal triangulations of manifolds of the form $F \times I$,  where $F$ is a closed, orientable surface and $I$ is a closed interval. The required results on minimal triangulations of manifolds with boundary in \S\ref{sec:main results} are of independent interest.


\subsection{Examples}
\label{sec:examples of inflations}

We begin by describing the construction of a fairly simple triangulations of $F \times I$.   Our triangulations come from the Jaco-Rubinstein inflation construction \cite{JR:Inflate} and are obtained by taking the cone over a minimal triangulation of a closed surface, then inflating at the ideal vertex created by the cone point.  We give a brief review of the inflation construction as needed for these examples.  Inflations of more general ideal triangulations and their inverse operation of crushing a triangulation along a normal surface are fully developed by Jaco-Rubinstein in \cite{JR:Inflate} and \cite{JR}, respectively.  

\subsubsection{Inflations of triangulations }
Suppose $\tri_g$ is a minimal triangulation of the closed, orientable surface $F_g$ of genus $g$ and let $\tri^*$ be the cone on $\tri_g$ with cone point $v^*$. An inflation of $\tri^*$ at $v^*$ is a triangulation $\tri$ of $F_g\times I$. The triangulation $\tri$ is very closely related to $\tri^*$; in particular, the inflation $\tri$ is a minimal vertex triangulation of $F\times I$ (has all its vertices in the boundary and only one vertex in each boundary component) and can be crushed along a component of its boundary \cite{JR:Inflate} giving back the triangulation $\tri^*$.  

The collection of all normal triangles at the vertex $v^*$ in the tetrahedra of $\tri^*$  form a normal surface (made up only of triangles) called the vertex-linking surface at $v^*$; let $S^*$ denote this vertex-linking surface. The surface $S^*$ has an induced triangulation, say $\mathcal{S}^*$, isomorphic to $\tri_g$; hence, $\mathcal{S}^*$ is a minimal triangulation of the vertex-linking surface $S^*$.  The minimal triangulation $\mathcal{S}^*$ of $S^*$ can be viewed as a triangulation of a $4g$--gon in the plane, obtained without adding vertices, and with its boundary edges identified in pairs.  The inflation construction starts with the selection of a minimal spine in the one-skeleton of the triangulation  $\mathcal{S}^*$; we call such a minimal spine a {\it frame}. In the current situation the collection of boundary edges in any $4g$--gon representation of $\mathcal{S}^*$ gives rise to a frame, say $\lambda$; such a frame $\lambda$ in $\mathcal{S}^*$ has one vertex and $2g$ edges. See Figure~\ref{f-local-pic-edge}.

An inflation of $\tri^*$ is guided by such a frame $\lambda$ in the triangulation $\mathcal{S}^*$ of the vertex-linking surface $S^*$. Each edge of $\lambda$ is a normal arc in a face of $\tri^*$ and corresponds to the intersection of that face with the vertex-linking surface $S^*$; see Figure \ref{f-local-pic-face}. Each vertex in $\lambda$ corresponds to the intersection of an edge of $\tri^*$ with the vertex-linking surface $S^*$. Figure \ref{f-local-pic-edge} shows examples of possible frames in $\mathcal{S}^*_g$ for $g=1$ and $g=2$, along with their intersection with a small neighbour of the vertex of the frame in the vertex-linking surface $S^*$.  The frames are  indicated in the figure by bold edges.         
\begin{figure}[htbp]
         \begin{center}
\includegraphics[width=2.5in]{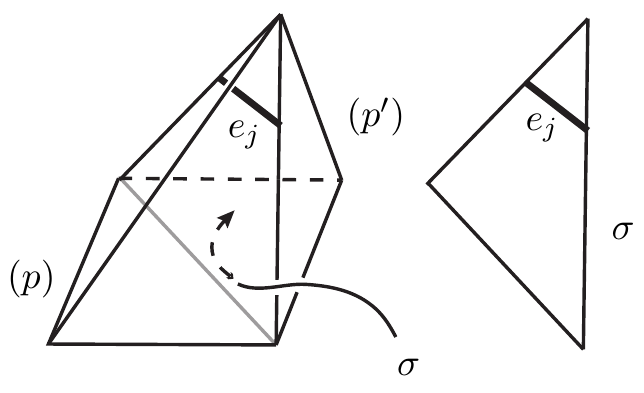}
\caption{The intersection of the frame with a face $\sigma$ of $\tri^*$ is a single edge of the frame; shown here as the bold line.} 
\label{f-local-pic-face}
\end{center}
\end{figure}

\begin{figure}[t]
  \begin{center}
    \subfigure[Local picture in $\tri_1$]{
      \includegraphics{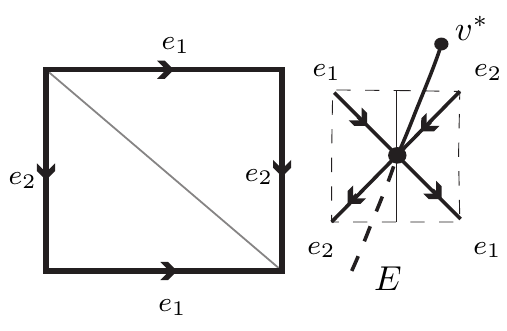}}
    \qquad
    \subfigure[Local picture in  $\tri_2$]{
      \includegraphics{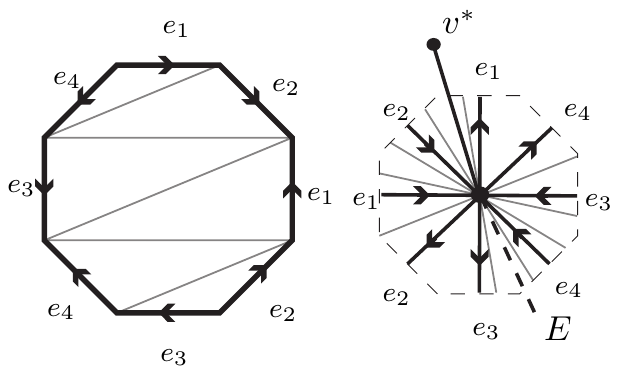}}
  \end{center}
\caption{The local view of edges of a frame at the vertex where the vertex-linking surface $S^*$  meets the edge $E$ of $\tri^*$. The bold edges are edges of the frame $\lambda$ and the lighter edges are non frame edges of the triangulation $\mathcal{S}^*$ of $S^*$, coming from the triangulations $\tri_1$ and $\tri_2$, respectively.}
\label{f-local-pic-edge}
\end{figure}

{\bf Inflation at a face.}  Each edge in a frame accounts for the addition of a tetrahedron to the ideal triangulation $\tri^*$ by the construction we call ``an inflation at a face" of $\tri^*$.  This construction comes with a prescription for undoing face identifications of $\tri^*$ and introducing new face identifications between faces of tetrahedra in $\tri^*$ and faces of the added tetrahedra; see Figure \ref{f-blow-up-face1}. At this step some of the faces of the added tetrahedra have not been assigned face identifications. In Figure \ref{f-blow-up-face1}(B) these are the faces $(e_j)(012)$ and $(e_j)(013)$; this is resolved by a construction at each vertex of the frame, which we refer to as ``an inflation at an edge of $\tri^*$." The new edge $(e_j)(01)$ will be an edge in the triangulation $\tri$ that is in the boundary of $F_g\times I$; the edge  $(p)(cb)\leftrightarrow (e_j)(23)\leftrightarrow (p')(c'b')$ is an edge of the given triangulation $\tri_g$ of $F_g$.

\begin{figure}[t]
  \begin{center}
    \subfigure[$(p)(abc)\leftrightarrow (p')(a'b'c')$]{
      \includegraphics[height=3.5cm]{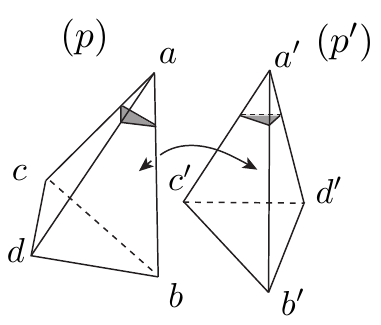}}
    \qquad
    \subfigure[$(p)(abc)\leftrightarrow(e_j)(032)$ \newline and $(e_j)(132)\leftrightarrow(p')(a'b'c')$]{
      \includegraphics[height=3.5cm]{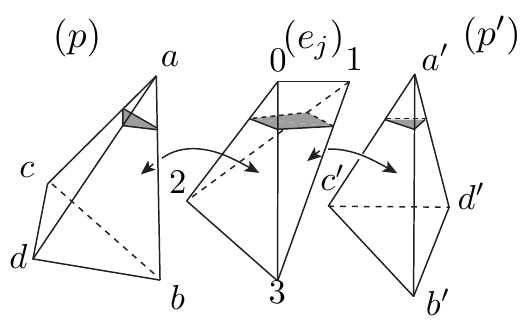}}
  \end{center}
  \caption{The edge $e_j$ of the frame $\lambda$ lies in the face common to tetrahedra $(p)$ and $(p')$. An inflation at the face $\sigma$ opens up the identification between $(p)$ and $(p')$ and adds a new tetrahedron $(e_j)$ along with induced face identifications  between $(p)$ and $(e_j)$ and $(e_j)$ and $(p')$.} 
\label{f-blow-up-face1}
\end{figure}

{\bf Inflation at an edge.} In general, an inflation at an edge can take on a combination from three possibilities denoted in \cite{JR:Inflate} as {\it generic, crossing}, or {\it branch}. However, in our very simple situation in this work, there is only one vertex in the frame $\lambda$ and at that vertex we have a branch of index $4g$ leading to precisely $4g$ unidentified faces of added tetrahedra. We complete these face identifications by adding a cone over a $4g$--gon, $P^*$ in Figure \ref{f-branch-config};  the unidentified faces of the added tetrahedra, following the ``inflation at a face,"  are identified with the $4g$ triangular faces in the cone $P^*$. The identification of  the unidentified faces of the added tetrahedron $(e_j)$ are determined by the orientation of the edge $e_j $ in $S^*$.
To complete the triangulation we can subdivide the cone in numerous ways using $4g-2$ tetrahedra.  
\begin{figure}[htbp]
 \begin{center}
\includegraphics[width=14cm]{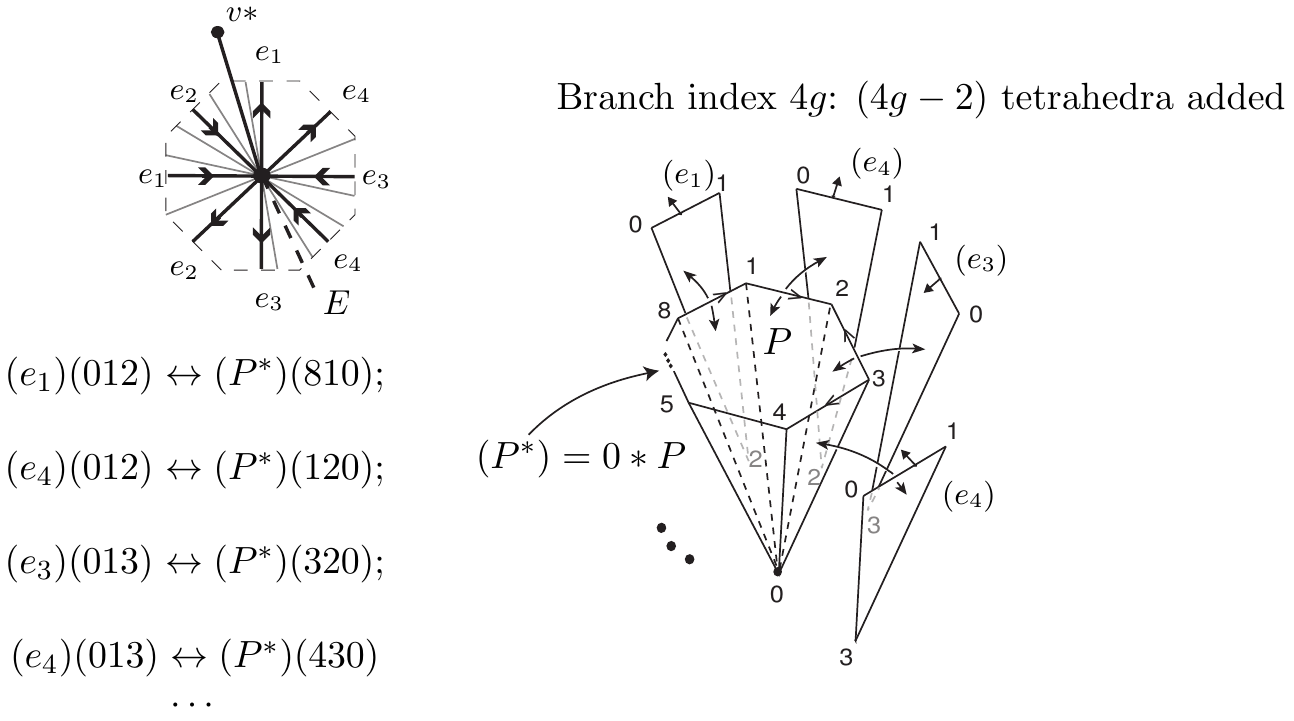} 
\end{center}
\caption{The edge $E$ meets the frame in one point, an index $4g$ branch point ($g\ge 1$). The cone over a planar $4g$-gon is added and then is subdivided into $(4g-2)$ tetrahedra, without adding vertices.} \label{f-branch-config}

\end{figure}

{\bf Complexity of an inflation.} The complexity of an inflation is defined in \cite{JR:Inflate}; the complexity is determined by the frame and, in general, involves the number of edges, the number of crossings, and the index of each branch point. However, the results from \cite{JR:Inflate} applied here give the complexity of our frame, written $C(\lambda)$, as $$C(\lambda) = e(\lambda) + 2(b(\lambda) - 1),$$ where $e(\lambda)$ is the number of edges of the frame $\lambda$ and $b(\lambda)$ is the index of the branch point of $\lambda$. It follows from \cite{JR:Inflate} that if $|\tri^*|$ denotes the complexity of the ideal triangulation $\tri^*$, then the complexity of the inflation  $\tri$ of $\tri^*$ is $$|\tri| = |\tri^*| + C(\lambda).$$  In particular, for $\tri^*$ the cone over a minimal triangulation of a compact, orientable surface of genus $g\ge 1$, $|\tri^*| = 4g-2$ and  $C(\lambda) = 2g+2(2g-1)$.  As a special case of Theorem 4.3 and Theorem 4.4 of \cite{JR:Inflate}, we have the following theorem.

\begin{theorem}\label{triang:inflation} Suppose $\tri_g$ is a minimal triangulation of the closed, orientable surface $F_g, g\ge 1$, and $\tri^*$ is the ideal triangulation formed by taking the cone over $\tri_g$ with ideal vertex, $v^*$. Let $\tri$ be an inflation of $\tri^*$. Then the underlying point set of $\tri$ is homeomorphic to $F_g\times I$ and $\tri$ has complexity $|\tri| = 10g-4.$
\end{theorem}

\subsubsection{Examples of Inflations}\label{sec:inflationexamples}

Here we give examples by carrying out the construction of Theorem \ref{triang:inflation}; these examples provide very straight forward examples of inflations of ideal triangulations and are shown later in this section to provide models for the minimal triangulations of the family of 3--manifolds $M= F_g\times I,$ where $F_g$ is the closed orientable surface of genus $g\ge 1$. We also provide a five tetrahedron (minimal) triangulation of $S^2 \times I$ via an inflation of the cone over a minimal triangulation of the 2-sphere, $S^2$. 
 
$\mathbf{(g=1)}$  In Figure \ref{F1 times I} we give an inflation of the ideal triangulation determined by taking the cone over a two-tetrahedron triangulation of $S^1\times S^1$. By following the sequence of steps we start with a minimal triangulation of the 2--torus, cone this triangulation getting an ideal triangulation of $S^1\times S^1\times [0,1)$ with ideal vertex $v^*$. We choose a frame $\lambda = <e_1>\cup<e_2>$ in the vertex-linking surface $\mathcal{S}^*$. The next step is to inflate in the edges of $\lambda$, adding the tetrahedra $(e_1)$ and $(e_2)$. We then inflate at the edge $E$ adding a cone on the 4--gon for the branch point of order 4. By Theorem~\ref{triang:inflation}, we have a  triangulation of $F_1 \times I$. Note the complexity of the frame is $C(\lambda) = 4$, hence, the 6 tetrahedron triangulation.

\begin{figure}[htbp]
 \begin{center}
\includegraphics[width=14cm]{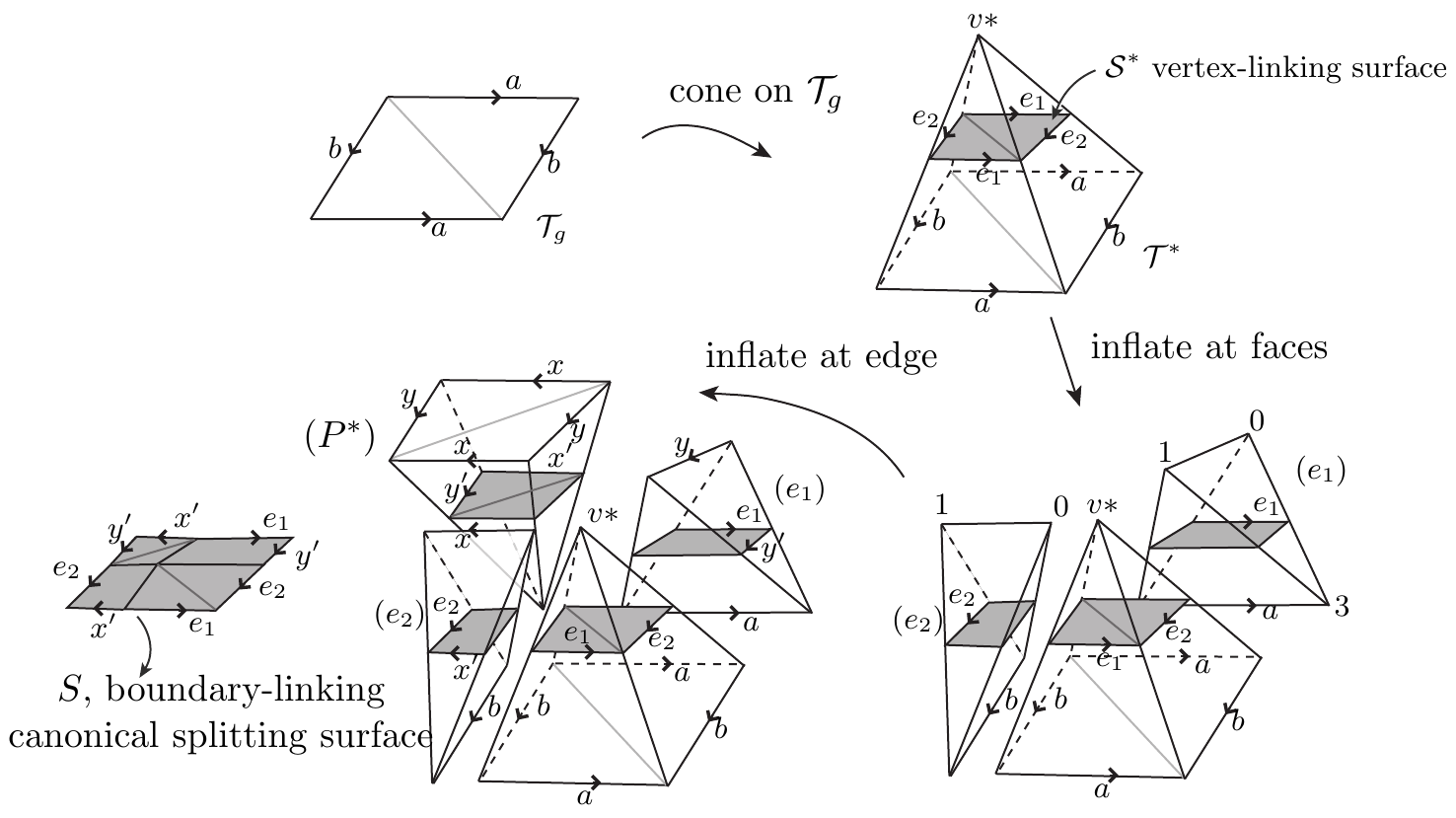} 
\end{center}
\caption{Triangulation of $F_g\times I$: An inflation of the cone on a minimal triangulation of $F_1 = S^1\times S^1$ at the cone point $v^*$, resulting in a 6 tetrahedra triangulation of $F_1\times I$.  The boundary-linking, canonical splitting surface $S$ is the inflation of the vertex-linking surface $\mathcal{S}^*$} \label{F1 times I}

\end{figure}

$\mathbf{(g\ge 2)}$  If $F$ has genus at least 2, the procedure is as above for the torus. In Figure \ref{fig:genustwo} we give an informative way to visualize an inflation of an idea triangulation by constructing the induced ``inflation" of the vertex-linking surface $S^*$.  The edges of the frame inflate to normal quadrilaterals in the tetrahedra added by an inflation at a face. The vertex of the frame inflates to $4g-2$ normal triangles in the $4g-2$ tetrahedra added with an inflation at an edge. This inflation of the vertex-linking surface $S^*$ results in a normal cell decomposition of a boundary-linking and canonical splitting surface, $S$, in $\tri$ of $F_g\times I$.

\begin{figure}
 \begin{center}
\includegraphics[width=8.5cm]{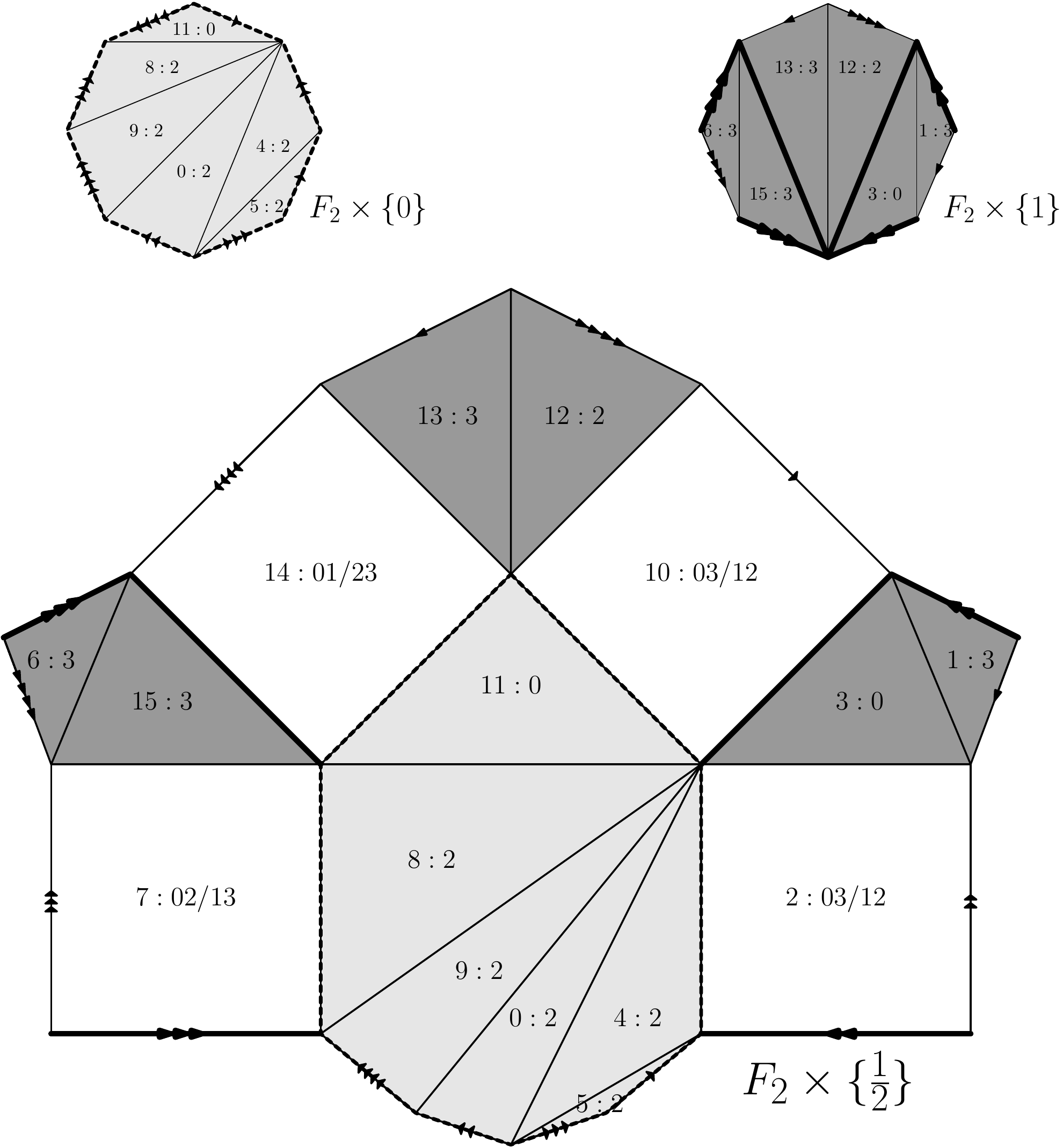}
	\caption{Boundary components and canonical separating genus two surface of a minimal $16$-tetrahedra triangulation
		of $F_2 \times I$. The solid and dashed lines denote spines in the boundary components
		corresponding to the four quadrilaterals in the splitting surface. \label{fig:genustwo}}
  \end{center}
\end{figure}

$\mathbf{(g = 0)}$
 In this example we show a minimal triangulation of $S^2\times I$ using the inflation construction. Note, in particular, that we have chosen a frame that separates and so for the inflation at the edge, we have two cones to attach, each a cone over a single triangle, itself a cone over the circle.  Again, the steps of the inflation construction can be observed by following the sequence of arrows, beginning with a 2-triangle minimal triangulation of $S^2$ and forming the cone over this triangulation with ideal vertex $v^*$. We have a frame with a single edge $\lambda = <e_1>$ in the vertex-linking surface $\mathcal{S}^*$. The next step is to inflate in the edge $e_1$, adding the tetrahedron $(e_1)$ to $\tri^*$ and a quadralateral to the vertex-linking surface $\mathcal{S}^*$. We then inflate at the edge $E$ adding two cones, each a cone on a triangle. By Theorem \ref{triang:inflation}, we have a 5--tetrahedron triangulation of $S^2 \times I$. This process is not unique, leading to three combinatorially distinct minimal triangulations of $S^2 \times I$. One of them is shown in the Figure~\ref{fig:s2xi} below. 

\begin{figure}[htbp]
 \begin{center}
\includegraphics[width= 6 in]{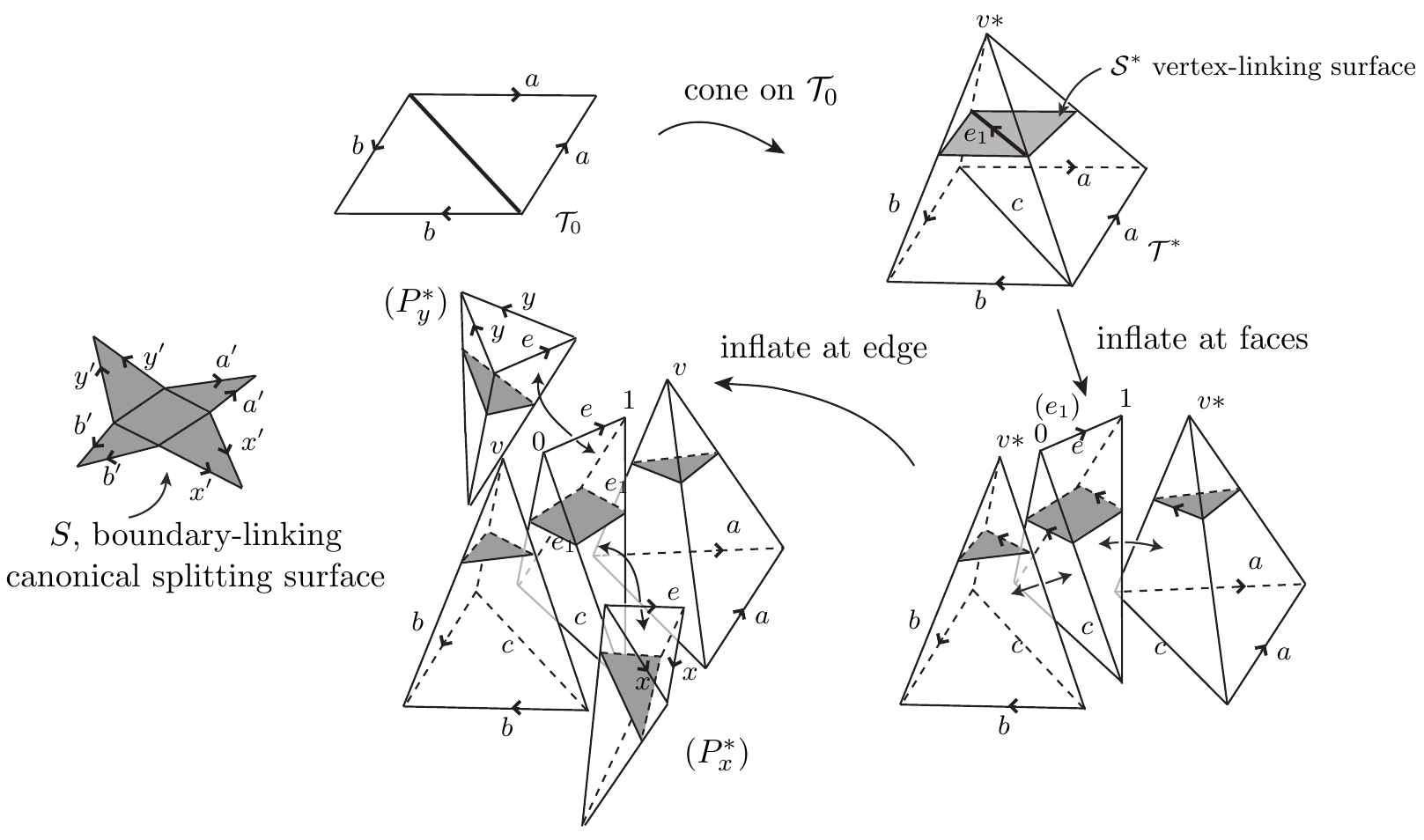} 
\caption{Triangulations of $S^2\times I$: A five-tetrahedron triangulation of $S^2\times I$ obtained by inflating from the cone point $v^*$ of the cone  on a minimal triangulation of $S^2$. \label{fig:s2xi}}
\end{center}
\end{figure}


\subsection{Boundary faces of tetrahedra and edges of small order}    
\label{sec:main results}

We start with some general observations. Suppose $M$ is a compact, irreducible, $\partial$--irreducible, orientable 3--manifold with non-empty boundary.
Jaco and Rubinstein \cite{JR} show that a minimal triangulation of $M$ is 0--efficient and all vertices are in $\partial M$ with precisely one vertex in each boundary component (unless $M$ is the 3--ball).  

\begin{lemma}\label{lem:one-face}
Suppose $M$ is a compact, irreducible, $\partial$--irreducible, orientable 3--manifold with non-empty boundary, and $M$ is not homeomorphic with the 3--ball. If $\mathcal{T}$ is a minimal triangulation of $M$, then 
no tetrahedron of $\mathcal{T}$  has more than one face in the boundary of $M$.
\end{lemma}

\begin{proof} Suppose $\mathcal{T}$ is a minimal triangulation of $M$. Our proof considers possible cases.

If a tetrahedron of $\mathcal{T}$ has four faces in the boundary, then $\mathcal{T}$ is the one-tetrahedron triangulation of the 3--ball, which contradicts $M$ not homeomorphic to the 3--ball.  

If a tetrahedron of $\mathcal{T}$ has three faces in the boundary, then $\mathcal{T}$ has a boundary component with at least two vertices since the vertex in common to the three faces cannot be identified with any of the other vertices of the tetrahedron.  Hence, by the results of \cite{JR} mentioned above, a component of $\partial M$ is a 2--sphere and contradicts $M$ not homeomorphic to the 3--ball. 

If a tetrahedron of $\mathcal{T}$ has two faces in the boundary, then suppose $\Delta$ is such a tetrahedron. In this case $\Delta$ has all its vertices in the same component, say $B$, of $\partial M$. Let $e$ be the edge of $\Delta$ common to the two faces of $\Delta$ in $B$.  Then $e$ is a diagonal of a quadrilateral $Q$ (possibly with some identifications) in a minimal triangulation of $B$ induced by the triangulation $\tri$.  Let $e'$ be the edge of $\Delta$ opposite $e$ (i.e., $\Delta$ is the join of the edges $e$ and $e'$). Suppose $\text{int}(e') \subset \text{int}(M)$. If the two faces of $\Delta$ having $e'$ in common are not identified with each other, then $\Delta$ is a tetrahedron layered on a triangulation of $M$  and hence, $\tri$ would not be a minimal triangulation of $M$. If the two faces of $\Delta$ having $e'$ in common are identified with each other, then $\tri$ is a one-tetrahedron triangulation of the 3--ball or the solid torus, contradicting our hypothesis. The only remaining possibility is $e'\subset\partial M$.  Hence, we have $e'\subset B$, the same component as $e$, and an edge in the minimal triangulation of $B$, say $\tri_B$, induced by $\tri$. However, $\e'$ is a loop in $M$ homotopic through $\Delta$ to the diagonal of $Q$ transverse to $e$. Let $e^*$ denoted the diagonal in $Q$ transverse to $e$; then a diagonal flip in $Q$ exchanging $e$ for $e^*$, gives a minimal triangulation of $B$ with both $e'$ and $e^*$ as edges. But $e'$ is homotopic through $M$ to $e^*$ and $M$ $\partial$--irreducible, gives that $e'$ and $e^*$ are homotopic in $B$. This is impossible for distinct edges of a minimal triangulation of $B$, unless $B$ were the 2--sphere.  But this contradicts $M$ is not homeomorphic to the 3--ball.
\end{proof}

\begin{remark} The hypotheses that $M$ not be homeomorphic to the 3--ball and that $M$ be $\partial$--irreducible are both necessary.  The minimal (one-tetrahedon) and 0-efficient triangulation of the 3--ball has two faces in the boundary. Every minimal triangulation of a handlebody of genus $g,$ $g\ge 1$, is layered and hence must have a tetrahedron with two faces in the boundary; a handlebody of genus $g\ge 1$ is not $\partial$--irreducible. 
\end{remark}

An analysis of edges of small order was provided for minimal triangulations of closed manifolds in \cite{JR,JRT}; we provide a similar analysis here, in the case of edges of order one or order two, for minimal triangulations of manifolds with boundary. The necessary modifications for the case of nonempty boundary from the argument in \cite{JR} are minor.

\begin{proposition} Suppose $M$ is a compact, orientable, irreducible, $\partial$--irreducible 3--manifold with nonempty boundary and $\tri$ is a minimal triangulation of $M$. \begin{enumerate}\item if $\tri$ has an edge of order 1, then $M$ is a 3--ball, \item if $\tri$ has an edge of order 2, then it must be in $\partial$M. \end{enumerate}
\end{proposition}

\begin{proof} Given $M$ as in the hypothesis, suppose $\tri$ is a minimal triangulation of $M$. 

Suppose $\tri$ has an edge $e$ of order 1. If $e$ is in $\partial M$, then a tetrahedron of $\tri$ would meet $\partial M$ in at least two faces; hence, by Lemma \ref{lem:one-face}, $M$ is homeomorphic to the 3--ball.  So, we consider the case where  the interior of $e$, an edge of order 1, is in $\text{int}(M)$.  Let $\Delta$ denoted the tetrahedron of $\tri$ containing $e$; then the edge $e'$ opposite $e$ in $\Delta$ bounds a disk in $\Delta\subset M$. If $e'$ were in $\partial M$, then by $M$ $\partial$--irreducible, we have $M$ a 3--ball and $\tri$ the one-tetrahedron, 0-efficient triangulation of the 3--ball. So, the only possibility is that $e'$ is in the interior of $M$.  Let $D'$ denote the disk bounded by $e'$ in $M$ and let $N = N(D')$ be a small regular neighborhood of $D'$ in $M$. Then the frontier of $N$, denoted $D$ is a properly embedded disk in $M$ ($D'$ meets $\partial M$ in the vertex of the triangulation $\tri$). The edge $e'$ serves as a barrier and $D$ shrinks to a normal disk or sweeps completely across $M$. The former contradicts $M$ 0--efficient  \cite{JR} and the latter results in $M$ being the 3--ball and contradicts $\tri$ minimal.

\begin{figure}[htbp]
 \begin{center}
\includegraphics[width=4in]{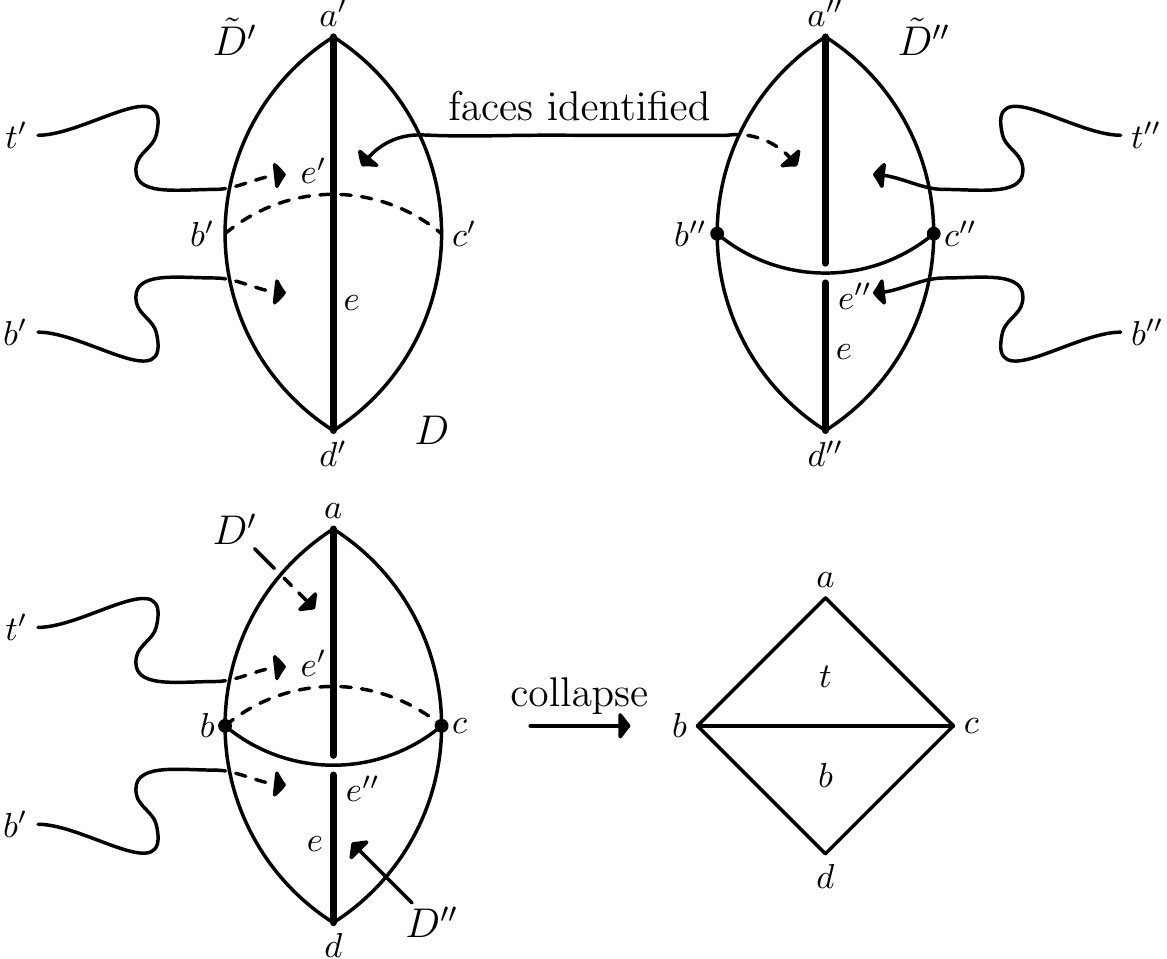}
	\caption{Interior edge of order two. \label{fig:order-two}}
 \end{center}
\end{figure}

Suppose $\tri$ has an interior edge, say $e$, of order 2. Then as in Figure \ref{fig:order-two} (Figure 38 of \cite{JR}), there are two tetrahedra $a'b'c'd'$ and $a''b''c''d''$ meeting in at least two faces that have the edge $e$ in common.  Our proof is similar to that in \cite{JR} with just a couple of minor twists. The idea, just as in \cite{JR}, is to show that the triangulation can be collapsed and contradict that it is minimal.  However, there are possible obstruction to such a collapse.  One obstruction would be that the edges $e'$ and $e''$ are already identified  and the identification takes $\overline{b'c'}$ to $\overline{c''b''}$; however, in this case there would be an $\mathbb{R}P^2$ embedded in $M$ contradicting $M$ irreducible ($M$ has nonempty boundary).  A second obstruction to collapsing is that the faces $t$ and $t'$ (or $b$ and $b'$) are both in $\partial M$; but then $M$ would have an isolated vertex $a$ (or $b$) in the boundary and $\tri$ would not be a minimal triangulation. The last possible obstruction to the desired collapse is that $t$ and $t'$ (or $b$ and $b'$) are  already identified.  Since $M$ is orientable, there are three possible identifications of the face $a'b'c'$ with the face $a''b''c.''$   Just as in the proof of Proposition 6.3 of \cite{JR}, two of these identifications lead to a 3--fold in $M$, giving $M$ a connected summand the lens space $L(3,1)$ and contradicting $M$ irreducible ($M$ has boundary); the third 
$a'b'c'\leftrightarrow a''b''c''$ makes the vertex an interior vertex contradicting $\tri$ minimal.
\end{proof}

Recall that the complexity of a 3--manifold was defined in \S\ref{prelim:complexity} as the minimal number of tetrahedra in a triangulation. 

\begin{proposition} Suppose $M$ is a compact, irreducible, $\partial$--irreducible, orientable 3--manifold with non-empty boundary. A lower bound on the complexity of $M$ in terms of the genus of its boundary is given by:
$$c(M) \ge 2 \sum_{F \subseteq \partial M} (2g(F)-1),$$
where the sum is taken over all connected components $F$ of $\partial M.$
\end{proposition}
\begin{proof} This is an immediate consequence of Lemma \ref{lem:one-face} and the fact that the number of faces in a minimal triangulation of a closed surface of genus $g$ is $4g-2$.
\end{proof}


\subsection{Complexity of $\mathbf{F\times I}$ and its minimal triangulations}
\label{subset:complexity of FxI}
\label{subsec:minimal triangulations of FxI}

\begin{theorem}\label{thm:complexity of FxI}
Let $F$ be a closed, orientable surface of genus $g\ge 1$ and $I$ be a closed interval.
Then $c(F\times I) = 10g-4.$
\end{theorem}

\begin{proof}
The lower bound on complexity from above gives $c(F\times I)\ge 8g-4,$ where $g=g(F).$ This arises from taking into account the boundary faces of the triangulation. Choose one colour (called red) for one boundary component, denoted $F_r$, and another colour (called blue) for the other, $F_b$. We have $\text{genus}(F_r)=\text{genus}(F_b)=g.$

This gives tetrahedra of types $A_r,$ $A_b,$ $B_r,$ $B_b,$ and $C,$ where $A_x$ has four vertices of colour $x,$ $B_x$ precisely three vertices of colour $x,$ and $C$ precisely two vertices of colour $x.$ The faces in the red boundary component are all in tetrahedra of type $A_r$ or $B_r,$ so $|A_r|+|B_r| \ge 4 g -2.$ Similarly $|A_b|+|B_b| \ge 4 g -2,$ and so $|A_r|+|A_b|+|B_r|+|B_b|\ge 8g-4.$

We get a canonical normal splitting surface $S$ from the colouring, which contains exactly one quadrilateral disc for each tetrahedron of type $C.$
Hence Corollary~\ref{cor:genus of boundary homologous} implies that $|C| = \Quad(S) \ge 2g.$ Thus, $c(F\times I)\ge 10g-4.$ The equality now follows from the triangulations described in \ref{subsec:minimal triangulations of FxI}.
\end{proof}


\begin{proposition}\label{pro:min S^2 x I}
$c(S^2\times I) = 5$
\end{proposition}

\begin{proof}
Above we provided an example of a triangulation of $S^2\times I$ with 5 tetrahedra; it remains to prove that this is the minimal number required. We will use the notation from the proof of Theorem~\ref{thm:complexity of FxI}. The splitting surface $S$ is not a vertex link and hence must contain at least one quadrilateral disc, whence  $|C|\ge 1.$ Also note that $S$ cannot consist of quadrilaterals alone, and that the number of triangles in $S$ is even. Since each quadrilateral disc has two short edges and two long edges, and each small (resp.\thinspace large) triangle has three short (resp.\thinspace long) edges, the numbers $|B_r|$ and $|B_b|$ are both even. It remains to show that neither can be zero. We may take advantage of the symmetry of the situation and suppose that $|B_r|=0.$ Since $F_r$ is non-empty, this implies $|A_r|\neq 0,$ but then the triangulation is not connected since tetrahedra of type $A_r$ can only connect to tetrahedra of type $C$ through tetrahedra of type $B_r.$ In particular, the minimal number required is $|C|=1$ and $|B_r|=2=|B_b|.$ It is now easy to check that any triangulation of this form arises from the inflation procedure and that there are precisely three combinatorially inequivalent minimal triangulations.
\end{proof}

We conclude that the triangulations constructed in \S\ref{sec:examples of inflations} are minimal triangulations; we end this discussion by showing that all minimal triangulations arise in this way. 

\begin{proposition}
Let $S$ be a closed, orientable surface. Every minimal triangulation of $F \times I$ is obtained by inflating the cone over a 1--vertex triangulation of $F.$
\end{proposition}

\begin{proof}
We suppose that $g(F)\ge 1$ (the case of a sphere was already treated in Proposition~\ref{pro:min S^2 x I}),  and use the notation and conclusions from the proof of Theorem~\ref{thm:complexity of FxI}. Assuming minimality of the triangulation gives $|C| = 2g$ and $|A_r|+|A_b|+|B_r|+|B_b|= 8g-4.$ In particular, each tetrahedron of type $A_x$ or $B_x$ has a unique face in the boundary of $F \times I.$ But this forces $|A_r|=0=|A_b|,$ since otherwise the triangulation is not connected. So there are $2g$ tetrahedra of type $C$ and $4g-2$ tetrahedra of each type $B_r$ and $B_b.$ 

Denote $S$ the splitting surface, which is known to have genus $g.$ We first show that there are exactly two triangle regions in $S,$ each of which is a closed disc. Indeed, the vertex link $V_b$ of the blue vertex is a disc and $S$ has a subsurface normally isotopic to the subsurface of $V_b$ consisting of exactly those normal triangles in $V_b$ that do not meeting $\partial V_b.$ Similarly for the link of the red vertex. Moreover, this accounts for all normal triangles present in $S.$ Let $R_b$ and $R_r$ denote these triangle regions in $S.$ It follows that $\partial R_b$ (resp. $\partial R_r$) is a union of quadrilateral edges in $S.$

The homotopy taking $S$ into the blue boundary component takes $R_b$ onto $F_b$ and hence $\partial R_b$ onto a frame in $F_b.$ Since there are $2g$ quadrilateral discs, there are at most $2g$ edges in the frame on $F_b.$ But since the triangulation of $F_b$ is minimal, there must be exactly $2g$ edges in the frame and in particular, no two quadrilaterals in $S$ meet along blue edges. It follows that the disc $R_b$ in $S$ has $4g$ edges and $4g-2$ faces, and hence is a $4g$--gon triangulated with all vertices on the boundary. The same reasoning applies to the disc $R_r.$ 

Crushing the triangulation of $F \times I$ along $S$ now results in two triangulated cones; one is a cone on the triangulation of $F_b,$ and the other a cone on the triangulation of $F_r.$ We work with the former triangulated cone and denote it $\tri_b.$ The frame on $F_b$ arising from $\partial R_b$ can be isotoped to a frame in the link of the cone point of $\tri_b$. Now inflation inserts $2g$ tetrahedra of type $C$ as well as a cone on a $4g$--gon. The interior of the $4g$--gon is naturally identified with the image of the interior of $R_r$ on $F_r,$ and hence the cone can be subdivided to give a triangulation combinatorially equivalent to the one we started with. 
\end{proof}


\section{The polyhedral realisation problem}
\label{sec:polReal}

The famous {\em realisation problem} asks whether or not a given combinatorial orientable surface $S$ is {\em realisable}, i.e.
if it has a polyhedral embedding into $\mathbb{R}^3$. We will show that there is no sequence of realisable normal surfaces of {\em 
unusually high genus}, i.e. super-linear genus with respect to the number of vertices.


\subsection{Polyhedral embeddings}

A decomposition of a surface into polygons such that the intersection 
of each pair of polygons is either empty, a common vertex or a common 
edge is called a {\em combinatorial surface}. Given a combinatorial surface 
$S$ with set of vertices $V$, a {\em polyhedral 
embedding} of $S$ is a function
$$ i \co V \hookrightarrow \mathbb{R}^3 $$
assigning coordinates to the vertices of $S,$ such that the convex hull of the 
vertices of each polygon of $S$ under $i$ 
is a (flat) $2$-dimensional polygon in $\mathbb{R}^3 $ and that 
two of these polygons intersect at most in a common vertex or edge, 
i.e.\thinspace no self-intersections of the surface occur.
An orientable combinatorial surface is called {\em realisable into $\mathbb{R}^3$} if it has a polyhedral 
embedding.


\subsection{Surfaces of unusually high genus}

Despite major research efforts to solve the realisation problem in general,
only partial results exist as of today. It is well-known
that all combinatorial $2$-spheres are realisable
\cite{Steinitz06UberEulPolyederRel}. Moreover,
while there are {\em combinatorial} tori with no polyhedral embedding
\cite{Ziegler08PolSurfsOfHighGenus,Gruenbaum03ConvPoly},
we know that all the {\em triangulated} ones are realisable
\cite{Archdeacon07How2ExToroidalMapsInSpace}.
For higher genus surfaces results become increasigly sparse.
For example, we know due to \cite{Schewe10NonrealMinTrigGenus6Surf} 
that none of the neighbourly 
$12$--vertex triangulated orientable 
surfaces of genus six, i.e. the ones with the maximum number of edges,
is realisable.

Furthermore, many other aspects of the realisation problem are largely unknown.
Here we will focus on one of them: given a sequence of combinatorial surfaces
$(S_k)_{k \in \mathbb{N}}$ with $n_k$ vertices, $n_k \overset{k \to \infty} \to \infty$, 
what is the highest genus $g(S_k),$ such that $S_k$ can be polyhedrally embedded into $\mathbb{R}^3$?
An elementary calculation shows that a
combinatorial surface $S$ with $n$ vertices has genus at most quadratic in $n$, namely
$ g(S)~\leq~\frac{1}{12}~(n-3)~(n-4)$
(see for example 
\cite[Lemma 2.1]{Ziegler08PolSurfsOfHighGenus}) and this bound is 
sharp in infinitely many cases as shown by 
Ringel~\cite{Ringel74MapColThm}.
However, neither a general obstruction 
nor any examples are known for 
families of surfaces of genus 
$g(S_k) \sim n_k^{\alpha}$, $\alpha \in ( 1,2 ]$, 
to be realisable in $\mathbb{R}^3$. 
The best current lower bound of
$$ g(S_k) \sim \frac{1}{8} n_k \log_2 n_k $$
is due to McMullen, Schulz and Wills~\cite{McMullen83Pol2MfldsOfHighGenus} or, more recently, 
Ziegler~\cite{Ziegler08PolSurfsOfHighGenus}.
This bound is only slightly better than
the trivial bound $ g(S_k) \in O(n_k) $.
We will therefore call a sequence of 
realisable combinatorial surfaces $S_k$ to be of 
{\em unusually high genus} if
$\frac{g(S_k)}{n_k} \to \infty$.



\subsection{Polyhedral embeddings of normal surfaces}
\label{ssec:realTech}

There are various techniques to find realisable surfaces of high genus:

\begin{itemize}
	\item Cs\'asz\'ar~\cite{Csaszar49PolyWithoutDiags}  
		proved the realisability of M\"obius' 
		triangulated $7$--vertex torus by giving explicit 
		coordinates of the vertices, found by an
		{\em intuitive search}.
	\item Bokowski developed a more systematic approach to find coordinates 
		for a polyhedral embedding of a combinatorial surface using {\em oriented matroids}
		(see Bokowski and Brehm~\cite{Bokowski87PolyGenus3W10Vert}  or 
		Bokowski and Eggert~\cite{Bokowski91TouteRealisToreMoeb}). 
		This technique also yields obstructions to polyhedral embeddings of
		certain combinatorial 
		surfaces.
	\item The realisation problem was settled for small triangulations of surfaces
		using a computer aided random search for coordinates due to
		Lutz~\cite{Lutz08EnumRandRealTrigSurf}, Bokowski~\cite{Bokowski08HeurMethFindRealSurf} 
		and Hougardy, Lutz and Zelke~\cite{Hougardy10SurfRealIntersSegFunc}.
	\item McMullen, Schulz and Wills~\cite{McMullen83Pol2MfldsOfHighGenus} constructed
		a sequence of surfaces of unusually high genus by 
		recursively connecting parallel copies of a surface
		with itself by an increasing number of handles. 
		These are the first examples of realisable combinatorial surfaces 
		with less vertices than handles.
	\item A more recent approach, due to Ziegler \cite{Ziegler08PolSurfsOfHighGenus},
		looks at surfaces as sub-complexes of the 
		$2$--skeleton of a $4$--polytope. Any of these sub-complexes has a 
		polyhedral realisation by projecting the $4$--polytope into 
		one facet and hence into $\mathbb{R}^3$ via the projection of a 
		Schlegel diagram. This approach yields an alternative 
		construction of polyhedral embeddings of combinatorial surfaces 
		of unusually high genus. To illustrate the power and convenience
		of this method note that the $7$--vertex M\"obius torus is a subcomplex 
		of the cyclic $4$-polytope with $7$ vertices and hence realisable (cf. 
		Cs\'asz\'ar's result in \cite{Csaszar49PolyWithoutDiags} and 
		\cite{Altshuler71PolyhedralRealizationsTori,Duke70EmbeddingOfMoebiusTorus}).
		Another corollary of this method is that any triangulated surface
		(oriented or non-oriented) has a polyhedral embedding into 
		$\mathbb{R}^5,$ since any triangulated surface
		occurs as a subcomplex of a cyclic $6$-polytope.
\end{itemize}

\medskip
Here, we present an additional realisation technique
using {\em normal surfaces}. Namely, we consider a
normal surface $S$ in the boundary complex $\partial P$ of 
the simplicial $4$--polytope $P$. The surface
$S$ can be realised in $\mathbb{R}^{3}$ by projecting 
$S \subset \partial P$ into a tetrahedron $\Delta \in \partial P,$
which is disjoint to $S$. If no such tetrahedron exist,
we can choose $\Delta$ to be a tetrahedron with the least number
of normal discs and push these discs out of $\Delta$
by inserting a small number of extra vertices in their exterior:
for any normal triangle near a vertex $v$ of $\Delta$
place an extra vertex above $v$ and cone over its
three normal arcs, see Figure~\ref{fig:subdivTrig}. 
For each quadrilateral, enlarge the quadrilateral around 
$\Delta$ by adding three extra vertices and cone over a fourth 
vertex to close the surface as shown in Figure~\ref{fig:subdivQuad}.
Since at most one quadrilateral type exist, arbitrarily many
normal pieces inside $\Delta$ can be embedded simultaneously
in this way.

\begin{figure}[h]
	\begin{center}
		\includegraphics[width=0.5\textwidth]{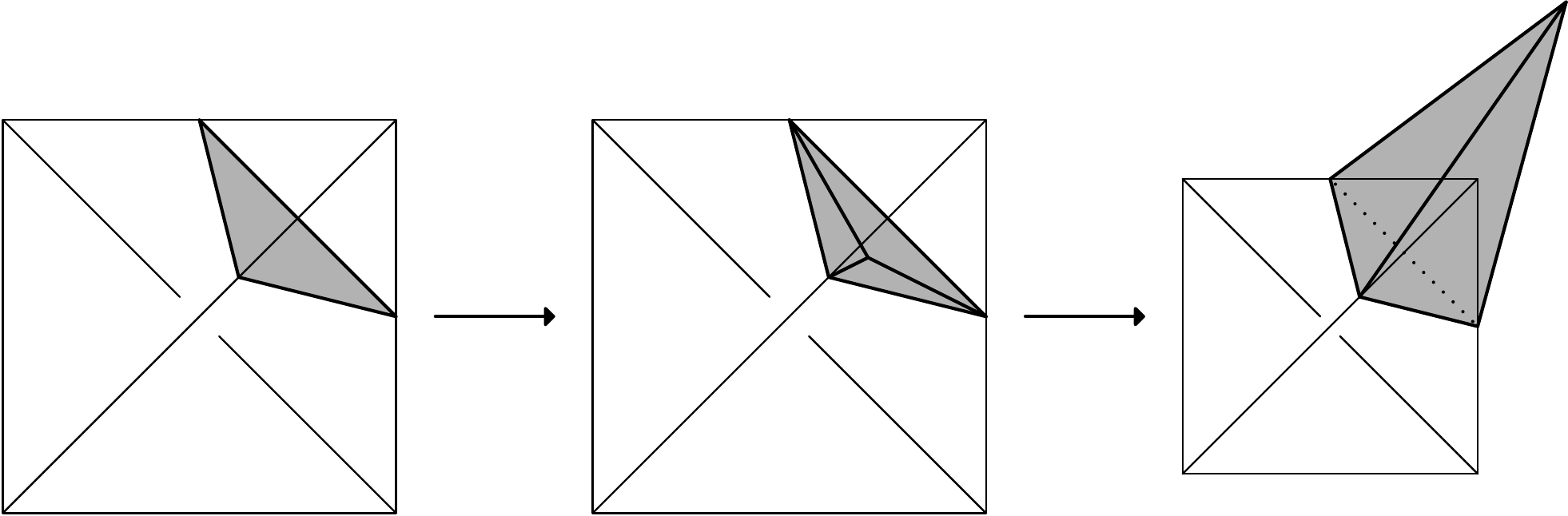}
	\end{center}
	\caption{Subdividing and embedding a normal triangle. \label{fig:subdivTrig}}
\end{figure}

\begin{figure}[h]
	\begin{center}
		\includegraphics[width=0.7\textwidth]{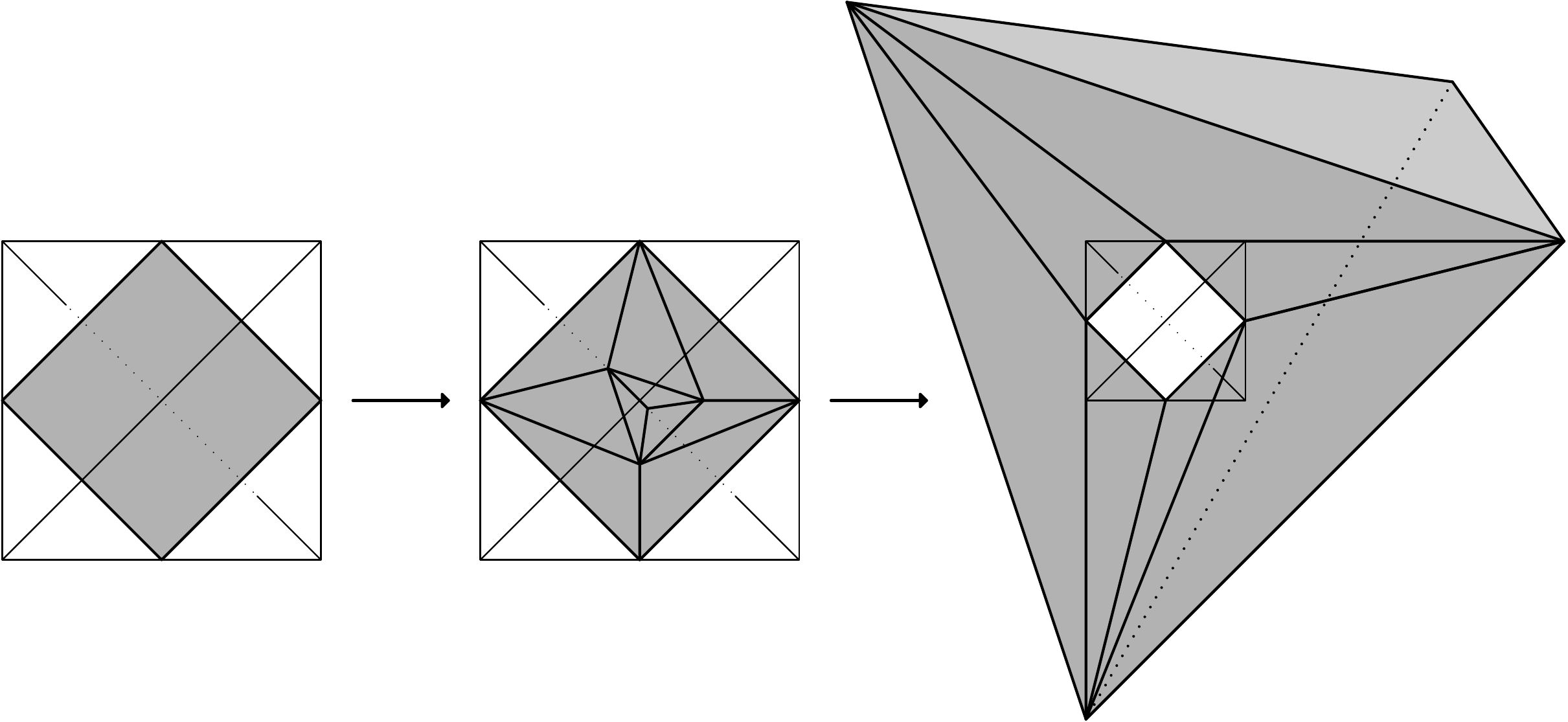}
	\end{center}
	\caption{Subdividing and embedding a normal quadrilateral. \label{fig:subdivQuad}}
\end{figure}

A normal surface where additional vertices have to be added will be referred to as \emph{nearly realisable}.
For instance, each Gale surface of Example~\ref{ex:GaleSli}
is contained in the boundary complex of the cyclic $4$--polytopes, and 
one needs to add four vertices in the interior of one quadrilateral disc. 
The Gale surfaces give a family of nearly realisable normal surfaces with 
increasing genus.

This new realisation technique applies to normal subsets of {\em all}
triangulations of $\mathbb{R}^3$ and thus a natural question
to ask is whether if it can be used to construct normal surfaces of
unusually high genus. We will use Corollary~\ref{cor:simplicial}
from Section~\ref{sec:genus bound} to show that 
this is in fact not possible.


\subsection{An obstruction to normal surfaces of unusually high genus}
\label{sec:slicings}

\begin{theorem}
	\label{thm:main}
	Let $M$ be a combinatorial 3--manifold and let $S \subset M$ be a closed, orientable normal surface.
	Then 
	$$ 2 g (S) < 7 f_0 (S). $$
\end{theorem}

\begin{remark}
The theorem in particular applies with $M = \partial P,$ where $P$ is a simplicial $4$-polytope, and $S \subset \partial P$ an orientable normal surface in the boundary complex of $P$.
\end{remark}

\begin{definition}[(vista from vertex link)]
	\label{def:vista}
	Let $M$ be a compact, triangulated $3$--manifold with vertex set $V.$ Let $x \in V$ and denote $V_x$ the normal surface in $M$ linking $x.$ Suppose $S \subset M$ is a normal surface. 
	Then the \emph{vista $C_x (S)$ of $S$ from $V_x$} is the union of all normal arcs contained in the quadrilateral subcomplex of $S,$ which are normally isotopic into $V_x.$
	\end{definition}
	
As an example of the definition, consider the normal sphere $S$ linking the edge $\langle x, y \rangle$ in a combinatorial 3--manifold. The vistas $C_x (S)$ and $C_y (S)$ are topological circles, whilst any other vista is either empty or an interval. The following lemma is proven by drawing a picture of the vista on the vertex link.

\begin{lemma}\label{lem:numEdges}
Let $M$ be a combinatorial 3--manifold with vertex set $V$ and $x \in V.$ Suppose $S \subset M$ is a normal surface in $M.$ If $C_x (S)$ is non-empty, then the number of edges in $C_x (S)$ is strictly less than $3$ times the number of vertices in $C_x (S)$.
\end{lemma}

\begin{proof}
Let $G$ be a connected component of $C_x (S).$ Then $G$ is a graph, and there is a normal homotopy taking it to a graph $G'$ in $V_x.$ Denote $p\co G \to G'$ the natural projection map. 
Now $G'$ is a planar graph, and since $M$ is a combinatorial 3--manifold, $G'$ is simple. So $G'$ is either a tree or has at least three vertices. In either case, 
$e(G')< 3v(G').$ If the pre-image of some edge of $G'$ contains $k$ edges in $G,$ then the pre-image of each of its endpoints contains at least $k$ vertices in $G$ (possibly more). In particular, for each vertex $v$ of $G',$ $|p^{-1}(v)| \ge \max \{ |p^{-1}(e)|\},$ where the maximum is taken over all edges $e$ incident with $v.$
Whence we also have $e(G)< 3v(G).$
%
\end{proof}

\begin{proof}[Proof of Theorem~\ref{thm:main}]
%
	Each vertex $v$ of the quadrilateral subcomplex of $S$ lies in exactly two vistas, being the two endpoints of
	the edge $S$ intersects in $v$.
	Whence 
	$$\sum \limits_{x \in V}  f_0( C_x (S) ) \le 2f_0 (S).$$

	By definition, no two vistas have an edge in common. Since there are at least $2q$ edges in the quadrilateral subcomplex of $S,$ we have
	$$ | \bigcup \limits_{x \in V}  C_x (S) | = 
	\sum \limits_{x \in V}  | C_x (S) | \ge 2 q . $$
	By Lemma~\ref{lem:numEdges}, we have
	$$ 3 f_0 (C_x (S)) > |C_x (S) |, $$
	and altogether, evoking Corollary~\ref{cor:simplicial}, we have
	\begin{eqnarray}
		6 g (S) 	&\leq&	7 q(S) \nonumber \\
			&\leq&	\frac{7}{2} \sum \limits_{x \in V}  | C_x (S) | \nonumber \\
			&<&	\frac{7}{2}\cdot 3 \cdot \sum \limits_{x \in V}  f_0( C_x (S) ) \nonumber \\
			&\leq&	7 \cdot 3 \cdot f_0 (S).\nonumber
	\end{eqnarray}
	This proves the result.
\end{proof}

\begin{remark}
	By \cite[Conjecture 4.6]{Spreer10NormSurfsCombSlic} 
	the inequality of Theorem~\ref{thm:main} would even change to
	$$ g (S) < f_0 (S) $$
	which is fairly close to the values of the Gale surfaces of 
	Example~\ref{ex:GaleSli}.
\end{remark}

\begin{remark}
	For each Gale surface of Example~\ref{ex:GaleSli} 
	its genus is close to the maximum genus possible for a (simplicial) normal 
	surface with the same number of vertices. 
	Since it is embedded in the boundary complex of a $4$--polytope, we see that
 	in the framework of normal surfaces, polyhedral 
	realisation into $\mathbb{R}^3$
	does not seem to invoke any significantly stricter 
	constraints than the existence 
	of a polyhedral embedding into an 
	arbitrary combinatorial $3$--manifold.
\end{remark}
	
\begin{question}
	Do similar observations hold for the
	polytopal subcomplex method? In particular,
	is it more difficult to find
	surfaces of unusually high genus
	in the $2$--skeleton of $4$--polytopes than	
	in the $2$--skeleton of an arbitrary combinatorial
	$3$--manifold?
\end{question}

%
%
%
%



\section{Examples}
\label{sec:Examples}

	\begin{figure}[h]
		\begin{center}
			\includegraphics[width=.8\textwidth]{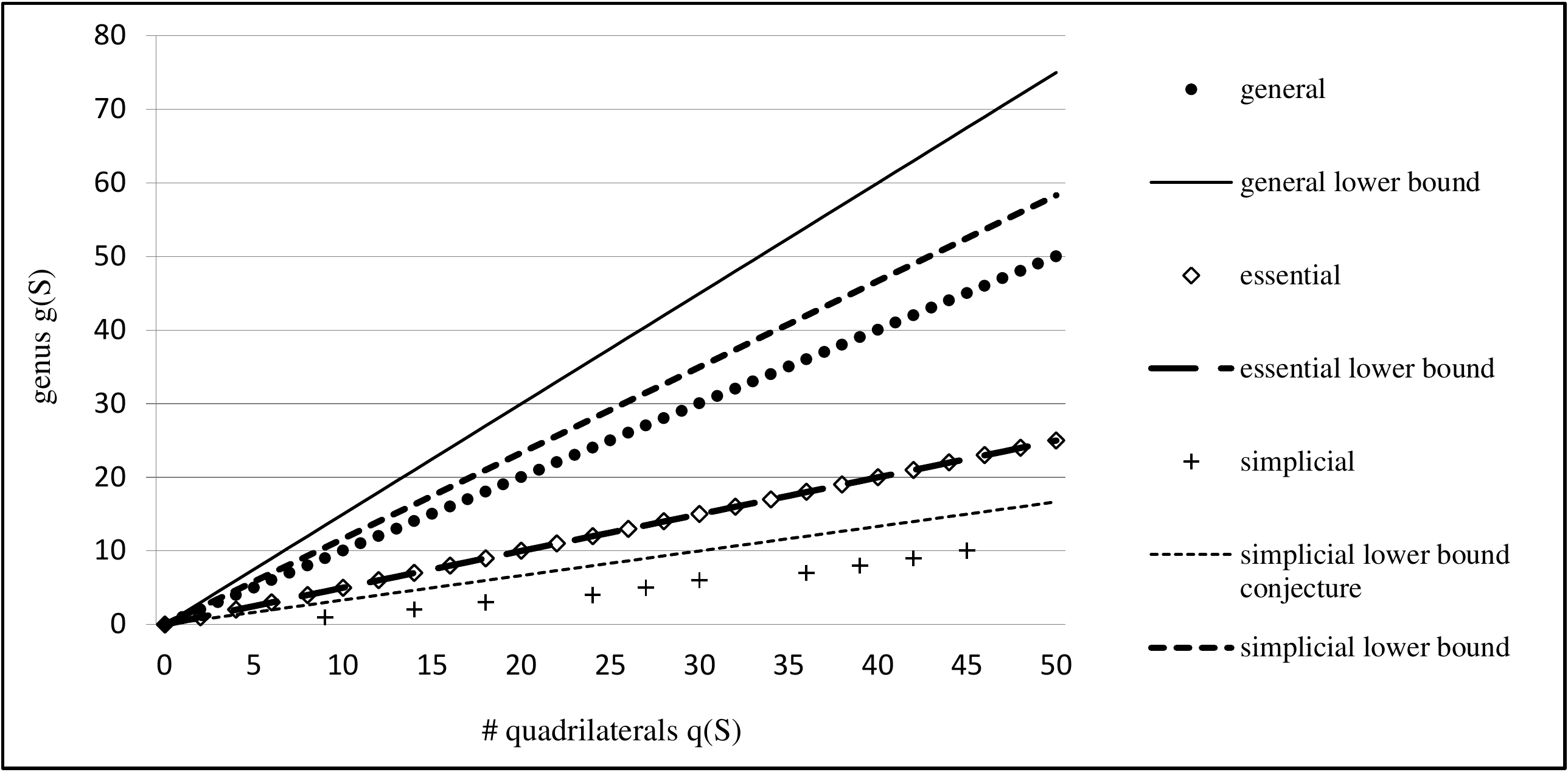}
		\end{center}
		\caption{Lower bounds and examples for simplicial, essential and general normal surfaces.
		\label{fig:graph}} 
	\end{figure}

In the following we will provide an extended set of examples certifying that most bounds presented in Section~\ref{sec:genus bound}
are in fact sharp. For an overview over the lower bounds in the simplicial, the essential and
the general case as well as the currently known best examples for genus $g(S) \leq 10$ see Figure~\ref{fig:graph}.

\subsection{Quadrilateral surfaces}

\begin{example}[(A Heegaard torus in $S^3$)]
	\label{ex:quadEx1}
	Lemma~\ref{lem:inequ for quad surf} is sharp for $v=1$ and genus $1.$ To see this consider the $1$--quadrilateral torus 
	in the $1$--tetrahedron triangulation of the $3$-sphere
	shown in Figure~\ref{fig:oneQuadTorus}. 
\end{example}	

\begin{figure}[h]
		\begin{center}
			\includegraphics[width=0.3\textwidth]{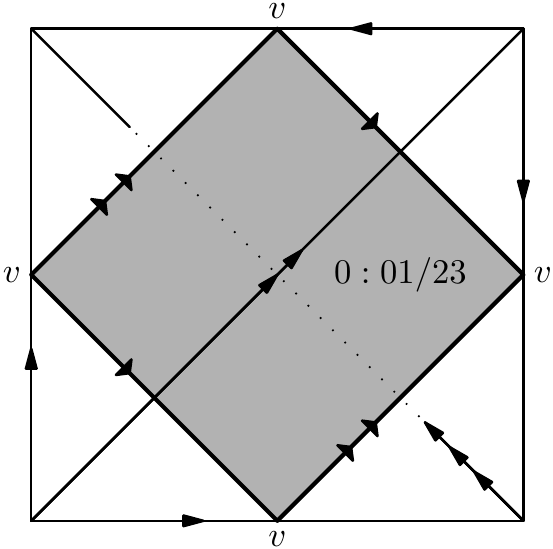}
		\end{center}
		\caption{An edge link in the $1$-tetrahedron triangulation
		of the $3$-sphere---a torus consisting of a single quadrilateral. 
		\label{fig:oneQuadTorus}}
	\end{figure}

\begin{example}[(Quadrilateral surfaces with two vertices)]
	\label{ex:quadEx2}
	Lemma~\ref{lem:inequ for quad surf} is also sharp for $v=2$ and arbitrary genus. 
	For the $2$-vertex case consider the following family of 
	$2 g$-tetrahedra $3$-vertex $3$-spheres $\mathscr{B}_g$, $g \geq 2$, given
	by the following gluing table; here $i(abc)$ in row $j$, column $(def)$ means that triangle 
	$(def)$ of tetrahedron $j$ is glued
	to triangle $(abc)$ of tetrahedron $i$. 
	Each of the $\mathscr{B}_g$, $g \geq 2$, contains a 
	quadrangulated genus $g$ splitting surface with only two vertices 
	for all $g \geq 2$.
	The splitting surface is given by one quadrilateral per tetrahedron
	each separating vertices $0$ and $2$ from $1$ and $3$
	and is shown in Figure~\ref{fig:quadEx}. 

	\medskip
	\begin{center}
	\begin{tabular}{|r|r|r|r|r|}
		\hline	
		tetrahedron&face $(012)$&face $(013)$&face $(023)$&face $(123)$\\
		\hline	
		\hline	
		$0$&$0(032)$&$1(013)$&$0(021)$&$1(123)$\\
		\hline	
		$1$&$2(012)$&$0(013)$&$2(023)$&$0(123)$\\
		\hline	
		$2$&$1(012)$&$\ldots$&$1(023)$&$\ldots$\\
		\hline	

		$\ldots$&$\ldots$&$\ldots$&$\ldots$&$\ldots$\\
		$\ldots$&$\ldots$&$\ldots$&$\ldots$&$\ldots$\\
		\hline	

		$2k-1$&$2k(012)$&&$2k(023)$&\\
		\hline	
		$2k$&$2k-1(012)$&$2k+1(013)$&$2k-1(023)$&$2k+1(123)$\\
		\hline	
		$2k+1$&$\ldots$&$2k(013)$&$\ldots$&$2k(123)$\\
		\hline	

		$\ldots$&$\ldots$&$\ldots$&$\ldots$&$\ldots$\\
		$\ldots$&$\ldots$&$\ldots$&$\ldots$&$\ldots$\\
		\hline	

		$2g-3$&$2g-2(012)$&$\ldots$&$2g-2(023)$&$\ldots$\\
		\hline	
		$2g-2$&$2g-3(012)$&$2g-1(013)$&$2g-3(023)$&$2g-1(123)$\\
		\hline	
		$2g-1$&$2g-1(032)$&$2g-2(013)$&$2g-1(021)$&$2g-2(123)$\\
		\hline	
	\end{tabular}
	\end{center}

	\begin{figure}[h]
		\begin{center}
			\includegraphics[width=0.8\textwidth]{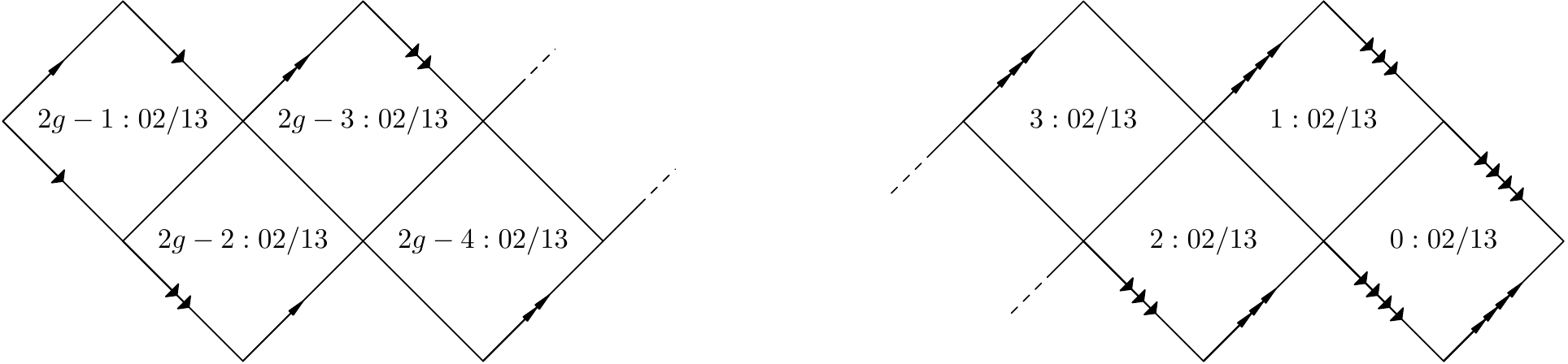}
		\end{center}
		\caption{Genus $g$ splitting surface in $\mathscr{B}_g$
		with only $2g$ quadrilaterals. \label{fig:quadEx}}
	\end{figure}
\end{example}

\begin{example}[(Gale surfaces)]
	\label{ex:GaleSli}
	The following construction describes a family of simplicial triangulations of 
	the 3--sphere containing interesting quadrilateral surfaces.
	The cyclic $4$--polytope $C_4(n)$ with vertex set 
	$V = \{ 0 , 1 , \ldots , n-1 \}$ is neighbourly, i. e. 
	it contains all possible
	${n \choose 2}$ edges. 
	On the other hand, by Gale's 
	evenness condition, the span of all odd vertices 
	$V_1 = \{ 1, 3, \ldots , n-1 \}$ 
	as well as the span of all even vertices 
	$V_2 = \{ 0, 2, \ldots , n-2 \}$ both have dimension $1$. 
	Altogether, the normal surface $G_n$ separating
	$V_1$ from $V_2$ slices each 
	of the ${n \choose 2} - n$ tetrahedra in the boundary of the polytope in a quadrilateral
	and has genus ${\frac{n}{2}-1 \choose 2}$.
	As a consequence, $G_n$ has $\frac{n^2}{4}$vertices and the $f$-vector
	\small
		$$ f(G_n) = \left (\frac{n^2}{4}, 2 \left ( {n \choose 2} - n \right ), 
		{n \choose 2} - n \right ).$$
	\normalsize
	Whence
	\small
		$$ g(G_n) = \frac{n^2}{8} - \frac{3n}{4} + 1 = 
		\frac{f_0 (G_n)}{2} - \frac{3 \sqrt{f_0 (G_n)}}{2} + 1. $$
	\normalsize
	We call the splitting surface $G_n$ a {\em Gale surface} 
	(cf. \cite{Spreer10NormSurfsCombSlic}). It has the maximum genus 
	with respect to $n$ and linear genus with a relatively 
	large constant with respect to $f_0 (G_n)$. Furthermore, it is
	polyhedrally realisable in $\mathbb{R}^3$ by just adding four vertices in one quadrilateral disc
	(cf. Section~\ref{sec:polReal}).
\end{example}

\begin{example}[(Simplicial triangulations)]
	\label{ex:simplicial}
	Given an $n$-vertex simplicial $3$-manifold triangulation $M$ containing a complete graph $K_m$
	on $m$ vertices $m < n$ such that each edge in $K_m$ is of degree three in $M$. Then the boundary
	of a small neighborhood of $K_m$ is a normal surface $S$ of genus $g(S) = {m-1 \choose 2}$ consisting of 
	$q(S) = 3 {m \choose 2}$ quadrilaterals and possibly a large number of triangles. Given an arbitrarily
	large number of vertices such a simplicial triangulationn $M$ is easy to construct: take a collection of 
	cones over 
	simplicial $2$-sphere triangulations with sufficiently many degree three vertices and pairwise join these cones
	together by identifying the star around a degree three edge, and finally closing off the resulting
	complex. Note that this family can also be extended
	to all values $g \geq 0$, $g(S) \neq {m-1 \choose 2}$, by slightly modifying $M$. All together this results in 
	in normal surfaces $S$ with
	$$ q(S) = 3 g(S) + O(\sqrt{g(S)}) .$$
	Despite an extended computer search using the classification of simplicial $3$-manifold triangulations 
	with few vertices \cite{Altshuler74CombMnfFewVert,Lutz08ManifoldPage} and
	the library of $3$-manifold triangulations contained in the {\em GAP}-package {\em simpcomp} \cite{simpcomp} 
	these examples of normal surfaces 
	in simplicial triangulations have the least number of quadrilaterals for fixed $g \neq 2$ amongst all
	known examples. In the case $g=2$ there is an $8$-vertex $3$-manifold triangulation of $S^3$ containing
	a normal surface with only $14$ quadrilaterals where the above method would result in $15$ quadrilaterals.

	It is thus conjectured that for normal surfaces in simplicial triangulations
	$$ q(S) \geq 3 g(S) $$
	holds \cite{Spreer10NormSurfsCombSlic}.
\end{example}

\subsection{Fundamental normal surfaces}

\begin{example}[($3q=2g$)]
	\label{ex:genleq3}
	Theorem~\ref{thm:genus} is sharp for $g=0$ and  $g=3,$ and hence a sharp linear bound of genus in terms of quadrilateral discs.
\begin{itemize}
\item[]
	{\bf $g=0$}: For this case consider any vertex linking sphere in a closed 3--manifold.
\item[]
	{\bf $g=3$}: The $6$-tetrahedra triangulation of the $3$-sphere
	given by below gluing table contains a genus $3$ vertex normal surface (in standard
	coordinates) with only two quadrilaterals, which is
	shown in Figure~\ref{fig:genus3}. 

	\medskip
	\begin{center}
	\begin{tabular}{|c|c|c|c|c|}
		\hline	
		tetrahedra&face $(012)$&face $(013)$&face $(023)$&face $(123)$\\
		\hline	
		\hline	
		$0$&$0(301)$&$0(120)$&$2(023)$&$1(123)$\\
		\hline	
		$1$&$3(012)$&$2(103)$&$2(123)$&$0(123)$\\
		\hline	
		$2$&$4(012)$&$1(103)$&$0(023)$&$1(023)$\\
		\hline	
		$3$&$1(012)$&$5(013)$&$4(123)$&$4(023)$\\
		\hline	
		$4$&$2(012)$&$5(203)$&$3(123)$&$3(023)$\\
		\hline	
		$5$&$5(231)$&$3(013)$&$4(103)$&$5(201)$\\
		\hline	
	\end{tabular}
	\end{center}	

	\begin{figure}[h]
		\begin{center}
			\includegraphics[width=0.8\textwidth]{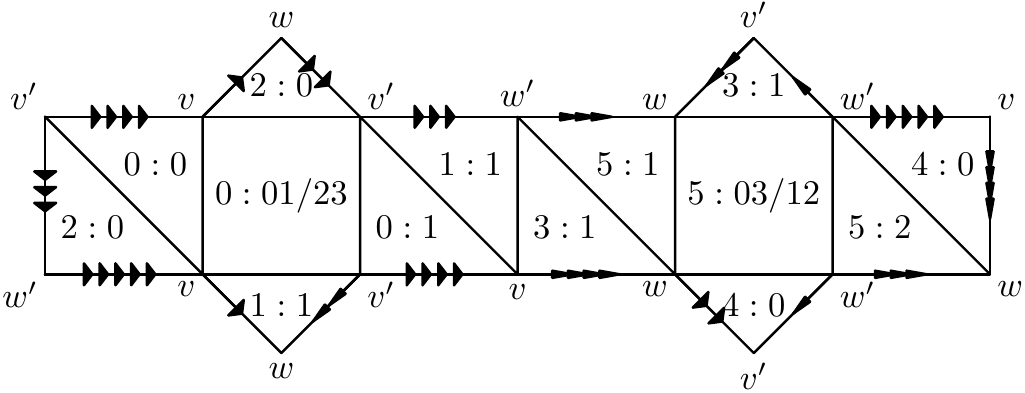}
		\end{center}
		\caption{A genus $3$ quadrilateral fundamental normal 
		surface with only $2$ quadrilaterals. \label{fig:genus3}}
	\end{figure}

\end{itemize}
\end{example}

\begin{example}[($q=g$)]
	\label{ex:q=g}
	For $g=1$ and $g=2$ there exist examples of orientable normal surfaces with the {\em minimal}
	number of quadrilaterals with respect to Theorem~\ref{thm:genus}.

\begin{itemize}
\item[]
	{\bf $g=1$}: For this case consider the $1$-quadrilateral
	torus inside the $1$-tetrahedron triangulation of the
	$3$-sphere from Example~\ref{ex:quadEx1}.
\item[]
	{\bf $g=2$}: There is a $4$-tetrahedra $0$-efficient 
	$3$-sphere given by the following gluings. 

	\medskip
	\begin{center}
	\begin{tabular}{|c|c|c|c|c|}
		\hline	
		tetrahedra&face $(012)$&face $(013)$&face $(023)$&face $(123)$\\
		\hline	
		\hline	
		$0$&$3(012)$&$1(013)$&$2(023)$&$1(123)$\\
		\hline	
		$1$&$3(013)$&$0(013)$&$2(013)$&$0(123)$\\
		\hline	
		$2$&$3(231)$&$1(023)$&$0(023)$&$3(023)$\\
		\hline	
		$3$&$0(012)$&$1(012)$&$2(123)$&$2(201)$\\
		\hline	
	\end{tabular}
	\end{center}
The triangulation contains a genus $2$ fundamental normal 
	surface in quadrilateral coordinates with only two quadrilaterals 
	which is shown in Figure~\ref{fig:genus2}.
\end{itemize}

	\begin{figure}[h]
		\begin{center}
			\includegraphics[width=0.6\textwidth]{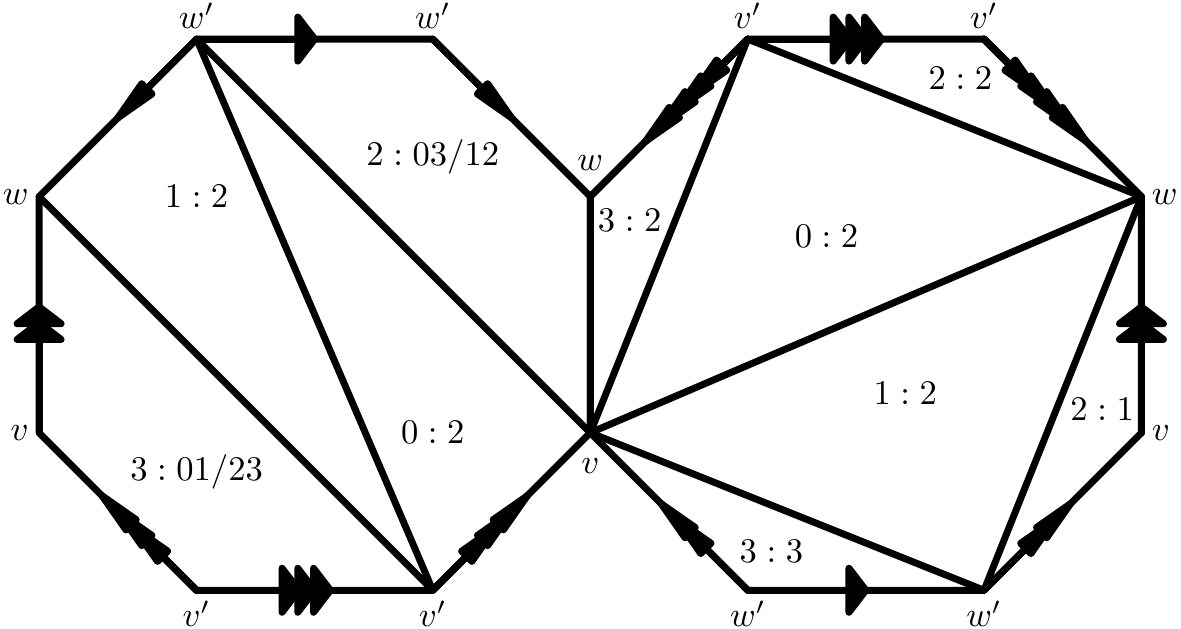}
		\end{center}
		\caption{A genus $2$ quadrilateral fundamental normal 
		surface with only $2$ quadrilaterals. \label{fig:genus2}}
	\end{figure}

\end{example}

\begin{example}[(Non-orientable normal surface)]
	\label{ex:nonor}
	The $5$-terahedra triangulation of $S^2 \times S^1$
	given by the gluing table

	\medskip
	\begin{center}
	\begin{tabular}{|c|c|c|c|c|}
		\hline	
		tetrahedra&face $(012)$&face $(013)$&face $(023)$&face $(123)$\\
		\hline	
		\hline	
		$0$&$0(013)$&$0(012)$&$2(023)$&$1(123)$\\
		\hline	
		$1$&$4(012)$&$4(013)$&$3(023)$&$0(123)$\\
		\hline	
		$2$&$4(230)$&$4(231)$&$0(023)$&$3(123)$\\
		\hline	
		$3$&$3(013)$&$3(012)$&$1(023)$&$2(123)$\\
		\hline	
		$4$&$1(012)$&$1(013)$&$2(201)$&$2(301)$\\
		\hline	
	\end{tabular}
	\end{center}

	\medskip
	contains a non-orientable vertex normal surface of
	Euler characteristic $-2$ with only $1$ quadrilateral.
	The surface is shown in Figure~\ref{fig:onequadgenus2}.
	Note that taking the orientable double cover
	of this surface yields a $2$ quadrilateral 
	orientable surface of Euler characteristic $-4$, and
	hence another sharp example for
	Theorem~\ref{thm:genus}.

	\begin{figure}[h]
		\begin{center}
			\includegraphics[width=0.3\textwidth]{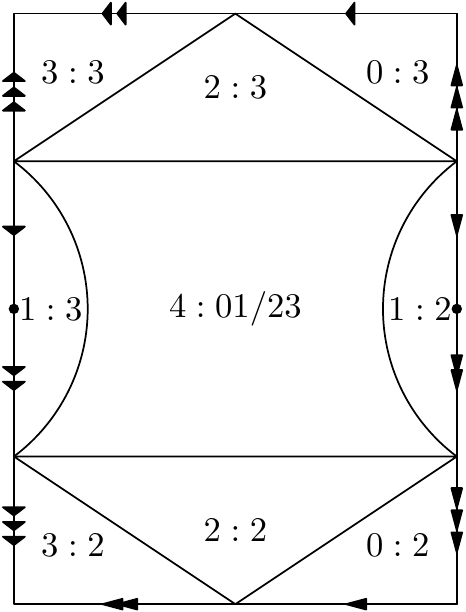}
		\end{center}
		\caption{A $1$ quadrilateral non-orientable normal surface 
		of Euler characteristic $-2$. \label{fig:onequadgenus2}}
	\end{figure}
\end{example}

\begin{example}[($g=q$)]
	\label{ex:family}
	The family of triangulations $\mathscr{A}_n$ given in 
	\cite{Burton12CompTopNormSurfExperiments} provides 
	examples of normal surfaces of arbitrary $g \geq 1$
	each with few quadrilaterals with respect to their genus.
	It consists of a cycle of $n$ tetrahedra, each identified 
	with itself around a degree one edge and 
	joined to each other along the remaining triangles 
	such that all vertices become identified.
	For all $n \geq 1$, the triangulation $\mathscr{A}_n$ 
	is a $1$--vertex $0$--efficient triangulation of the $3$--sphere.

	The triangulation $\mathscr{A}_n$ contains 
	exactly ${n \choose k}$ genus $k$ normal surfaces, 
	each having $k$ quadrilaterals dual to $k$ of 
	the $n$ edges of degree one. 
	In particular, there is a genus $n$ normal surface 
	containing exactly $n$ quadrilaterals.

	In standard coordinates all of these surfaces are 
	fundamental normal surfaces. However, the $n$ quadrilateral genus $n$ 
	normal surface is the sum 
	of $n$ fundamental tori minus $(n-1)$ copies of the vertex link; so it is not fundamental in $Q$--coordinates.
\end{example}

%
%
%

\begin{example}[(Normal surface with boundary)]
	\label{ex:bounded}
	There is a $4$-tetrahedra triangulation of the $3$-ball containing a 
	$2$-punctured torus with only one quadrilateral and thus 
	an example of equality in Corollary~\ref{cor:boundedSurfs}
	($g = 1$, $b = 2$, and $q = 1$).
	The triangulation
	is given by the following gluing table, where $\partial$ denotes a
	triangle in the boundary. The $2$-punctured torus is shown in 
	Figure~\ref{fig:bounded}.

	\medskip
	\begin{center}
	\begin{tabular}{|c|c|c|c|c|}
		\hline	
		tetrahedra&face $(012)$&face $(013)$&face $(023)$&face $(123)$\\
		\hline	
		\hline	
		$0$&$0(013)$&$0(012)$&$\partial$&$1(123)$\\
		\hline	
		$1$&$3(012)$&$3(013)$&$2(023)$&$0(123)$\\
		\hline	
		$2$&$2(013)$&$2(012)$&$1(023)$&$\partial$\\
		\hline	
		$3$&$1(012)$&$1(013)$&$3(312)$&$3(230)$\\
		\hline	
	\end{tabular}
	\end{center}

	\begin{figure}[h]
		\begin{center}
			\includegraphics[width=0.6\textwidth]{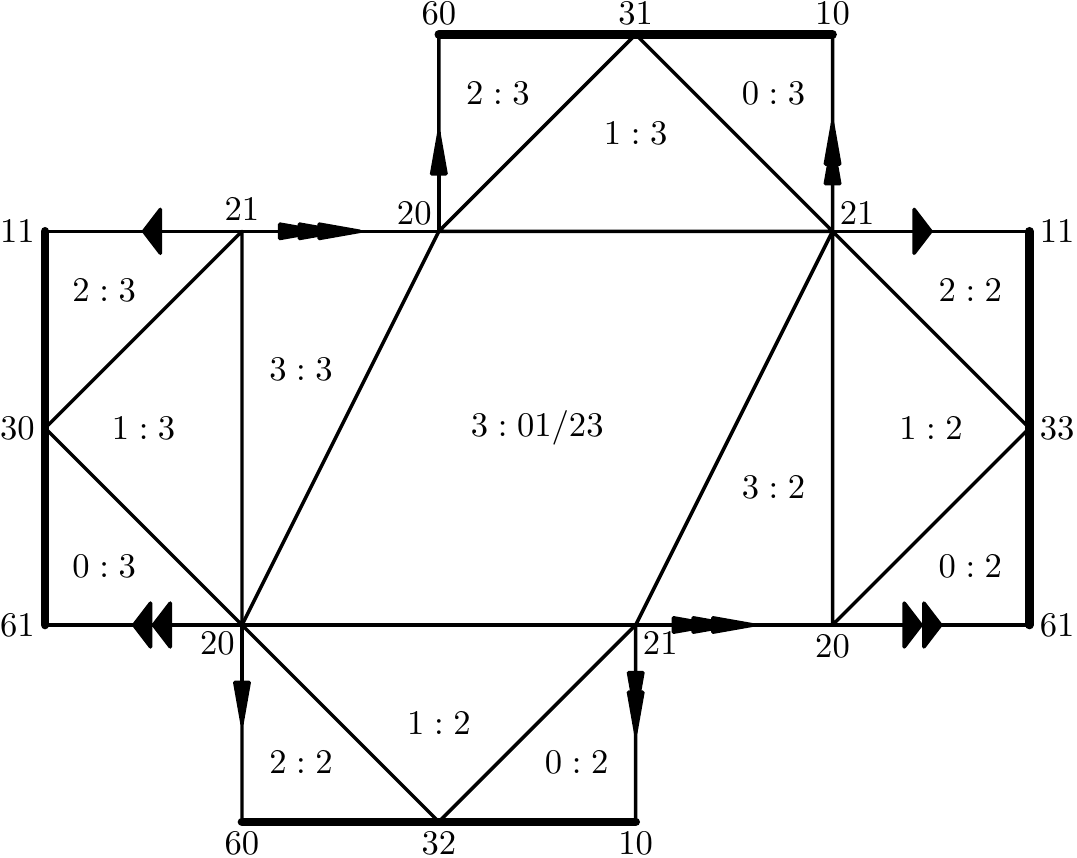}
		\end{center}
		\caption{A $2$-punctured torus ($b=2$, $g=1$) with only one quadrilateral.
		\label{fig:bounded}}
	\end{figure}
\end{example}

\begin{example}[(A non-trivial Haken sum with $g=5$ and $q=4$)]
	\label{ex:hakenSum}
	Consider the following $8$-tetrahedra triangulation $M$ of $S^2 \times S^1$:

	\medskip
	\begin{center}
	\begin{tabular}{|c|c|c|c|c|}
		\hline	
		tetrahedra&face $(012)$&face $(013)$&face $(023)$&face $(123)$\\
		\hline	
		\hline	
		$0$&$3(012)$&$1(013)$&$2(023)$&$1(123)$\\
		\hline	
		$1$&$4(132)$&$0(013)$&$4(023)$&$0(123)$\\
		\hline	
		$2$&$3(032)$&$5(013)$&$0(023)$&$3(321)$\\
		\hline	
		$3$&$0(012)$&$6(013)$&$2(021)$&$2(321)$\\
		\hline	
		$4$&$7(120)$&$7(013)$&$1(023)$&$1(021)$\\
		\hline	
		$5$&$6(032)$&$2(013)$&$6(021)$&$7(320)$\\
		\hline	
		$6$&$5(032)$&$3(013)$&$5(021)$&$7(123)$\\
		\hline	
		$7$&$4(201)$&$4(013)$&$5(321)$&$6(123)$\\
		\hline	
	\end{tabular}
	\end{center}

	\medskip
	$M$ contains a one quadrilateral non-orientable surface $\mathbf{s}_1$ of non-orientable genus
	$4$ with coordinates 
	\small
	$$	\begin{array}{rcl}
		\mathbf{s}_1&=(&(0, 0, 0, 0;0, 1, 0),(1, 0, 1, 0;0, 0, 0),(0, 0, 0, 1;0, 0, 0),(0, 1, 0, 0;0, 0, 0),\\
		&& (1, 1, 1, 0;0, 0, 0),(0, 0, 1, 1;0, 0, 0),(0, 1, 1, 0;0, 0, 0),(1, 1, 1, 0;0, 0, 0) \quad), 
		\end{array}
	$$
	\normalsize
	as well as a two quadrilateral orientable surface $\mathbf{s}_2$ of genus $3$
	with coordinates
	\small
	$$	\begin{array}{rcl}
		\mathbf{s}_2&=(&(0, 1, 0, 1;0, 1, 0),(1, 1, 1, 1;0, 0, 0),(0, 0, 0, 2;0, 0, 0),(0, 2, 0, 0;0, 0, 0), \\
			&& (1, 1, 1, 1;0, 0, 0),(0, 0, 0, 2;0, 0, 0),(0, 2, 0, 0;0, 0, 0),(1, 1, 0, 0;1, 0, 0) \quad).
		\end{array}
	$$
	\normalsize
	Hence, both $\mathbf{s}_1$ and $\mathbf{s}_2$ attain equality in Corollary~\ref{for:nonorSurfs}
	and Theorem~\ref{thm:genus} respectively. Moreover, $\mathbf{s}_1$ and $\mathbf{s}_2$ are
	compatible, non-disjoint and the sum of the orientable double cover of $\mathbf{s}_1$
	together with $\mathbf{s}_2$ is connected. It follows that $2 \mathbf{s}_1 + \mathbf{s}_2$
	is a non-trivial example for which the bound from Corollary~\ref{cor:genus Haken sum}
	is sharp.

	$2 \mathbf{s}_1 + \mathbf{s}_2$ has $48$ triangles, $4$ quadrilaterals and genus $5$. It is
	shown in Figure~\ref{fig:hakenSum}.

	\begin{figure}[p]
		\begin{center}
			\includegraphics[width=\textwidth]{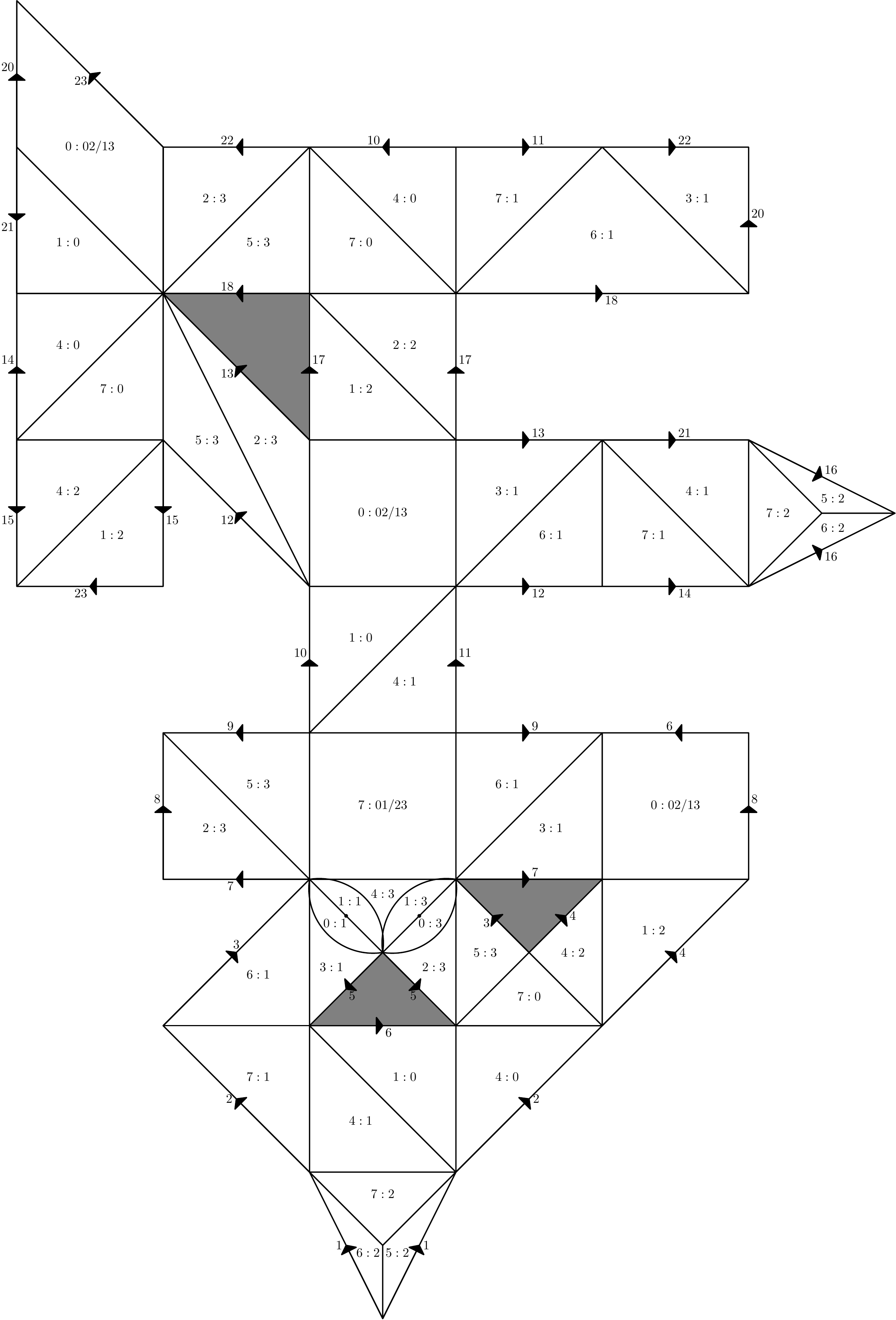}
		\end{center}
		\caption{A non-trivial Haken sum genus five surface with only four quadrilaterals. 
		\label{fig:hakenSum}}
	\end{figure}
\end{example}

\subsection{Incompressible surfaces}

Incompressible surfaces are extremely hard to find because checking for incompressibility requires running an algorithm with a doubly exponential worst case running time. However, the fourth author and others were able to obtain a complete classification of all incompressible surfaces of the $11031$ triangulations from the Hodgson-Weeks census by running extensive computer experiments \cite{Burton13EssSurfs}. As a result, there are $36449$ incompressible quad vertex normal surfaces, $19696$ of which are injective and $16753$ which occur as a double cover of a non-orientable vertex normal surface. All incompressible surfaces have genus $2 \leq g \leq 5$ and none of them attains equality in Corollary~\ref{cor:genus of essential}. However, some of them have only few more quadrilaterals than necessary. Comparing the number of quadrilaterals in these surfaces to the number of quadrilaterals in the complete set of $441331$ quad vertex normal surfaces of genus $2 \leq g \leq 5$ in the census reveals a slightly higher average number of quadrilaterals amongst the incompressible surfaces as can be seen in the following table.

	\medskip
	\begin{center}
	\begin{tabular}{|c||r|r|r|r||r|r|r|r|}
		\hline
		& \multicolumn{4}{|c||}{injective essential surfaces} & \multicolumn{4}{|c|}{double covers} \\
		\hline
		genus & \# surfaces & $\min_s q(S)$ & $\overline{q(S)}$ & $\max_S q(S)$&\# surfaces & $\min_S q(S)$ & $\overline{q(S)}$ & $\max_S q(S)$\\
		\hline	
		\hline	
		$2$&$ 16622 $& $ 6  $&$ 20 $&$ 90 $ &$ 11711 $& $ 6  $&$ 23 $&$ 114 $\\
		\hline	
		$3$&$ 3023 $& $ 15  $&$ 34 $&$ 109 $ &$ 4598 $& $ 14  $&$ 44 $&$ 226 $\\
		\hline	
		$4$&$ 46 $& $ 27  $&$ 48 $&$ 85 $ &$ 355 $& $ 24  $&$ 61 $&$ 166 $\\
		\hline	
		$5$&$ 5 $& $ 36  $&$ 43 $&$ 56 $ &$ 89 $& $ 32  $&$ 55 $&$ 98 $\\
		\hline	
		\multicolumn{9}{}{}\\
		\hline
		& \multicolumn{4}{|c||}{all essential surfaces} & \multicolumn{4}{|c|}{all (orientable) vertex surfaces} \\
		\hline
		genus & \# surfaces & $\min_S q(S)$ & $\overline{q(S)}$ & $\max_S q(S)$& \# surfaces & $\min_S q(S)$ & $\overline{q(S)}$ & $\max_S q(S)$\\
		\hline	
		\hline	
		$2$&$ 28333$ & $ 6  $&$ 21 $&$ 114 $  & $268202$ & $5$ & $13$ & $100$\\
		\hline	
		$3$&$ 7621$ & $ 14  $&$ 40 $&$ 226 $ & $120844$ & $7$ & $22$ & $187$\\
		\hline	
		$4$ &$ 401$ & $ 24  $&$ 59 $&$ 166 $ & $37877$ & $13$ & $33$ & $205$\\
		\hline	
		$5$&$ 94$ & $ 32  $&$ 55 $&$ 98 $  & $14408$ & $19$ & $42$ & $200$\\
		\hline	
		\hline
	\end{tabular}
	\end{center}

The complete data for all quad vertex normal surfaces compared to the incompressible surfaces of $2 \leq g \leq 4$ in the census is summarised 
by Figure~\ref{fig:plots}.

	\begin{figure}[h]
		\begin{center}
			\includegraphics[width=.49\textwidth]{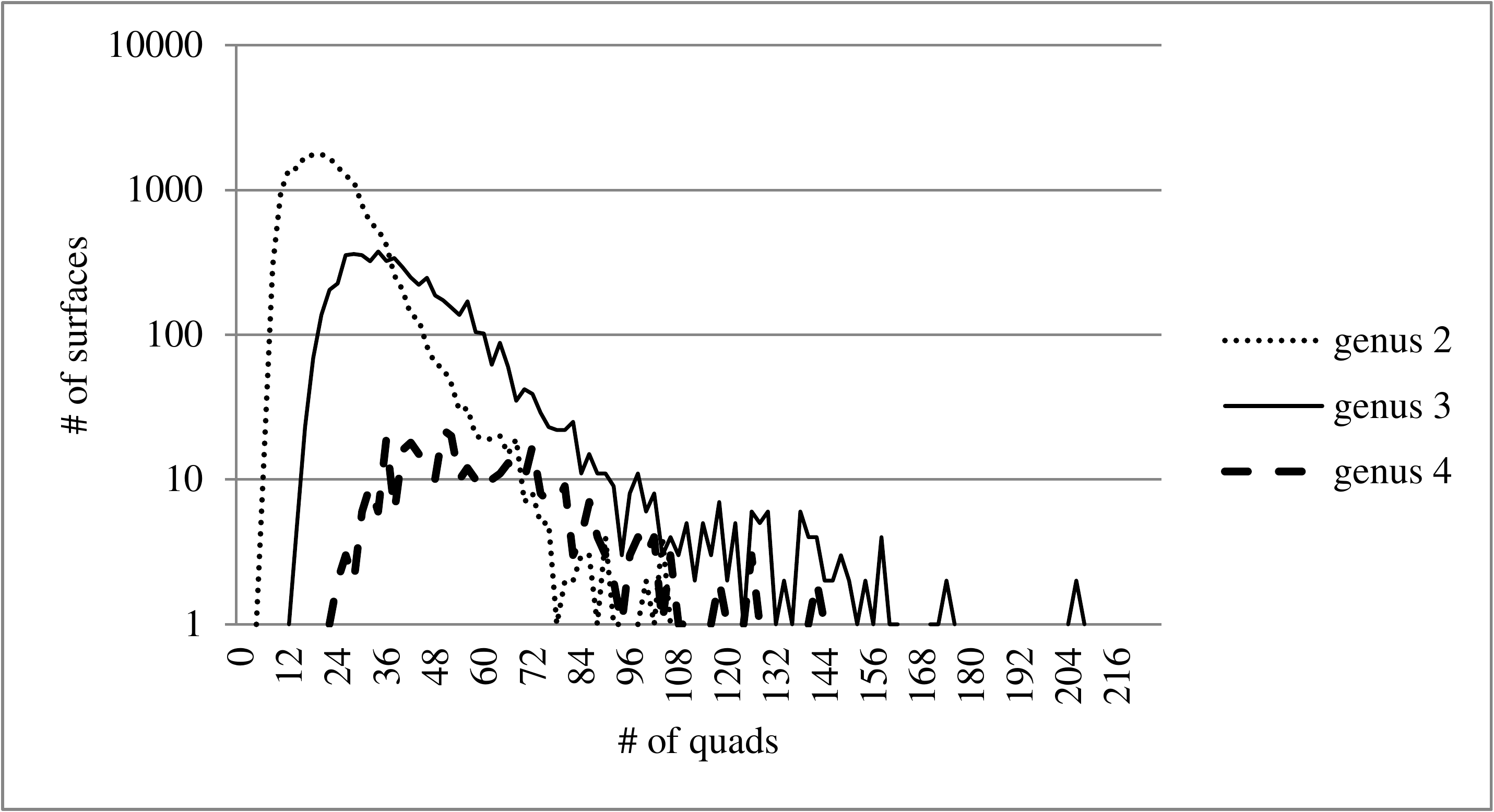} \hspace{.05cm} \includegraphics[width=.49\textwidth]{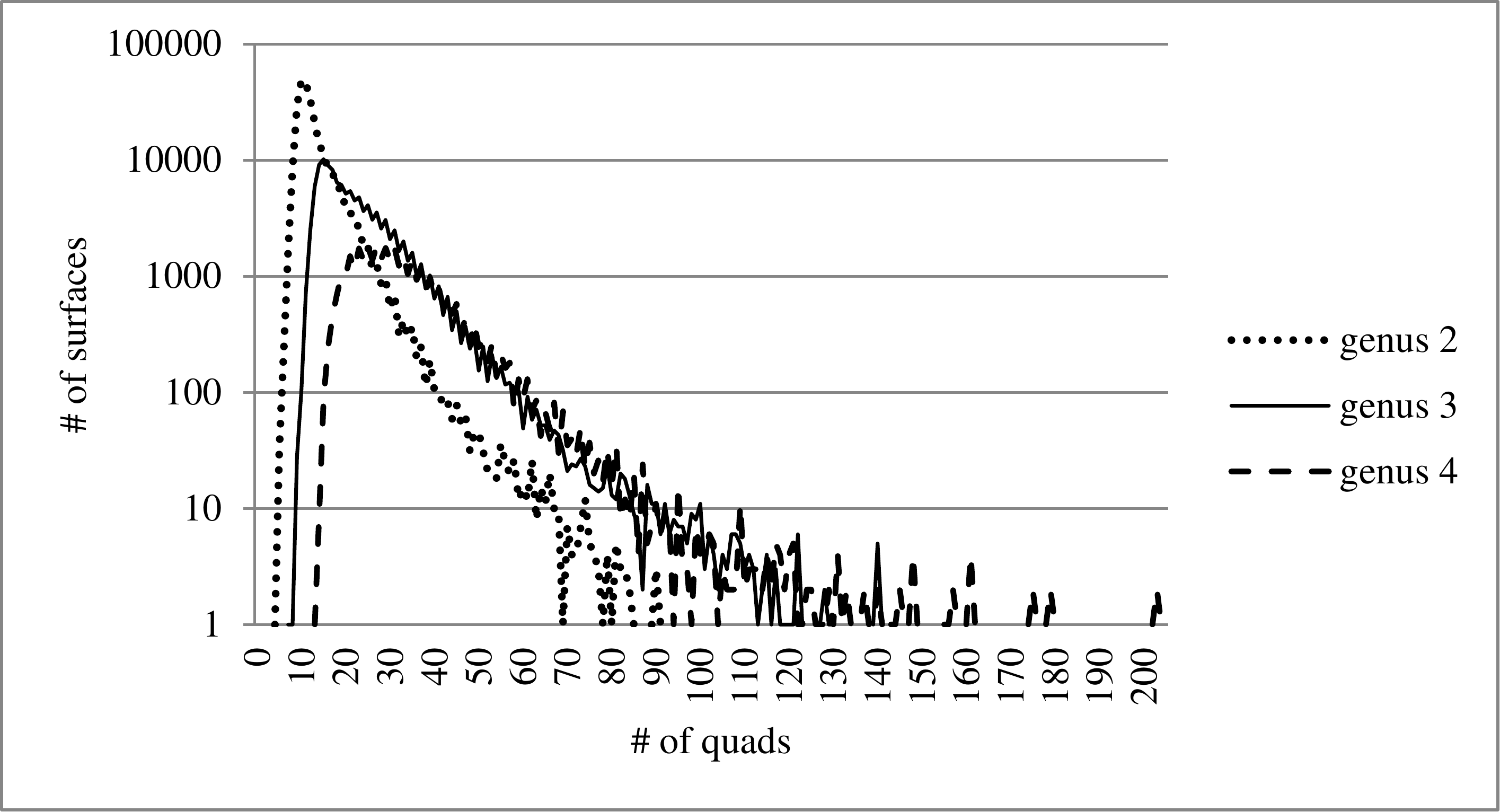}
		\end{center}
		\caption{Number of quads in the incompressible surfaces (left) and in all quad vertex normal surfaces (right) 
			of the Hodgson-Weeks census. \label{fig:plots}}
	\end{figure}

%

\begin{example}[(Incompressible surfaces with $q=2g$)]
	\label{ex:incompressible}
		By construction, the separating surfaces used in Section~\ref{sec:main results} to show that $c(S_g \times I) \leq 10g - 4$,
	are incompressible and have exactly $2g$ quadrilaterals each. Hence, the bound presented in Corollary~\ref{cor:genus of essential}
	is sharp for all values $g \geq 1$.

		Below we present the case $g = 1$ where the two quadrilateral incompressible torus lives inside  
	the minimal trivial torus bundle $T^2 \times S^1$ which can be obtained from the 
	six tetrahedra triangulation of $S_g \times I$ by identifying its two boundary components
	(cf. Figure~\ref{fig:ictorus}).

	\medskip
	\begin{center}
	\begin{tabular}{|c|c|c|c|c|}
		\hline	
		tetrahedra&face $(012)$&face $(013)$&face $(023)$&face $(123)$\\
		\hline	
		\hline	
		$0$&$4(012)$&$3(013)$&$2(023)$&$1(123)$\\
		\hline	
		$1$&$3(320)$&$4(230)$&$5(023)$&$0(123)$\\
		\hline	
		$2$&$3(231)$&$4(321)$&$0(023)$&$5(123)$\\
		\hline	
		$3$&$5(103)$&$0(013)$&$1(210)$&$2(201)$\\
		\hline	
		$4$&$0(012)$&$5(102)$&$1(301)$&$2(310)$\\
		\hline	
		$5$&$4(103)$&$3(102)$&$1(023)$&$2(123)$\\
		\hline	
	\end{tabular}
	\end{center}

	\begin{figure}[h]
		\begin{center}
			\includegraphics[width=0.6\textwidth]{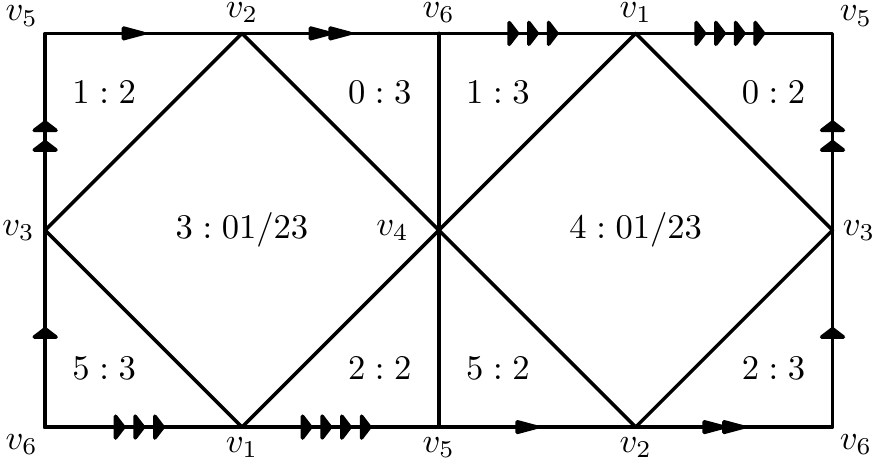}
		\end{center}
		\caption{A $2$ quadrilateral incompressible torus. 
		\label{fig:ictorus}}
	\end{figure}
\end{example}




\address{Department of Mathematics, Oklahoma State University, Stillwater, OK 74078-1058, USA}
\email{jaco@math.okstate.edu}

\address{Department of Mathematics, Oklahoma State University, Stillwater, OK 74078-1058, USA} 
\email{jjohnson@math.okstate.edu} 

\address{School of Mathematics and Physics, The University of Queensland, Brisbane, QLD 4072, Australia} 
\email{j.spreer@maths.uq.edu.au} 

\address{School of Mathematics and Statistics, The University of Sydney, Sydney, NSW 2006, Australia} 
\email{tillmann@maths.usyd.edu.au} 
\Addresses
                                                      

\begin{thebibliography}{99}

\bibitem{Altshuler71PolyhedralRealizationsTori}
A.~Altshuler.
\emph{Polyhedral realization in $R^{3}$ of triangulations of the torus
and $2$-manifolds in cyclic $4$-polytopes}.
Discrete Math., {\bf 1}, (1971/1972), no. 3, 211--238.

\bibitem{Altshuler74CombMnfFewVert}
A.~Altshuler.
\newblock {C}ombinatorial {$3$}-manifolds with few vertices.
\newblock {\em J. Combinatorial Theory Ser. A}, 16:165--173, 1974.

\bibitem{Archdeacon07How2ExToroidalMapsInSpace}
D.~Archdeacon, C.~P. Bonnington, and J.~A. Ellis-Monaghan.
\emph{How to exhibit toroidal maps in space}.
Discrete Comput. Geom., {\bf 38} (2007), no. 3, 573--594.

\bibitem{Bokowski08HeurMethFindRealSurf}
J.~Bokowski.
\emph{On heuristic methods for finding realizations of surfaces}.
In {\em Discrete differential geometry}, volume~38 of
Oberwolfach Semin., Birkh\"auser, Basel, (2008), 255--260.

\bibitem{Bokowski87PolyGenus3W10Vert}
J.~Bokowski and U.~Brehm.
\emph{A new polyhedron of genus $3$ with $10$ vertices}.
In {\em Intuitive geometry ({S}i\'ofok, 1985)}. Volume~48 of 
{\em Colloq. Math. Soc. J\'anos Bolyai}. North-Holland, Amsterdam,
(1987), 105--116.

\bibitem{Bokowski91TouteRealisToreMoeb}
J.~Bokowski and A.~Eggert.
\emph{Toutes les r\'ealisations du tore de {M}\"obius avec sept sommets}.
Structural Topology, {\bf 17}, (1991), 59--78.

\bibitem{BAB-2004} B.~A.~Burton: \emph{Face pairing graphs and 3-manifold enumeration}, J. Knot Theory Ramifications 13 (2004), no. 8, 1057--1101.

\bibitem{BAB-2007} B.~A.~Burton: \emph{Enumeration of non-orientable 3-manifolds using face-pairing graphs and union-find}, Discrete Comput. Geom. 38 (2007), no. 3, 527--571.

\bibitem{Regina} B.~A.~Burton, R.~Budney, W.~Pettersson, et al.: \emph{Regina: Software for 3-manifold topology and normal surface theory,}
http://regina.sourceforge.net/, 1999--2012. 


\bibitem{Burton13EssSurfs}
B.~Burton, A.~Coward, S.~Tillmann.
\emph{Computing closed essential surfaces in knot complements}.
In \emph{Proceedings of the 29th Annual Symposium on Computational Geometry},
(2013), 405--414.


\bibitem{burton10-tree}
B.~Burton and M.~Ozlen.
\emph{A tree traversal algorithm for
decision problems in knot theory and 3-manifold topology}.
Algorithmica, {\bf 65}(4), (2013), 772--801.

\bibitem{Burton12CompTopNormSurfExperiments}
B.~Burton, J.~Paixao and J.~Spreer.
\emph{Computational topology and normal surfaces: 
Theoretical and experimental complexity bounds}.
In {\em Proceedings of the 15th Meeting on Algorithm 
Engineering and Experiments}, (2013), 78--87.

\bibitem{TONS} D.~Cooper and S.~Tillmann: \emph{The Thurston norm via normal surfaces}.
Pacific J. Math. 239 (2009), no. 1, 1--15. 

\bibitem{Csaszar49PolyWithoutDiags}
{\'A}.~Cs{\'a}sz{\'a}r.
\emph{A polyhedron without diagonals}.
Acta Univ. Szeged. Sect. Sci. Math., {\bf 13}, (1949), 140--142.

\bibitem{Duke70EmbeddingOfMoebiusTorus}
R.~A.~Duke.
\emph{Geometric embedding of complexes}.
Amer. Math. Monthly, {\bf 77}, (1970), 597--603.

\bibitem{simpcomp}
F.~Effenberger and J.~Spreer.
\newblock simpcomp - a {GAP} package, {V}ersion 2.0.0.
\newblock \url{https://code.google.com/p/simpcomp}, 2013.

\bibitem{FV} E.~Fominykh and A.~Vesnin. 
\emph{Exact values of complexity for Paoluzzi--Zimmermann manifolds}
(Russian). Dokl. Akad. Nauk {\bf 439} (2011), no. 6, 727--729.
Translation in Dokl. Math. {\bf 84} (2011), no. 1, 542--544.

\bibitem{FMP} 
R.~Frigerio, B.~Martelli, C.~Petronio. 
\emph{Complexity and Heeagaard genus of an infinite class of compact 
3-manifolds}, Pacific J. Math. 210, (2003), 283--297. 

\bibitem{Gruenbaum03ConvPoly}
B.~Gr{\"u}nbaum.
\newblock {\em {C}onvex polytopes}. Volume 221 of 
{\em Graduate Texts in Mathematics}.
Springer-Verlag, New York, second edition, (2003).
Prepared and with a preface by Volker Kaibel, Victor Klee and
G{\"u}nter M.\ Ziegler.

\bibitem{Haken61TheoNormFl}
W.~Haken.
\emph{Theorie der Normalfl\"achen}.
Acta Math., {\bf 105}, (1961), 245--375.

\bibitem{Haken62HomeomProb3Mflds}
W.~Haken.
\emph{\"Uber das Hom\"oomorphieproblem der $3$–Mannigfaltigkeiten. I}, 
Math. Z., {\bf 80}, (1962), 89--120.

\bibitem{HLP}
J.~Hass, J.~C.~Lagarias, and N.~Pippenger.
\emph{The computational complexity of knot and link problems}. 
J. ACM {\bf 46}, (1999), no. 2, 185--211.

\bibitem{Hemion} G.~Hemion: \emph{On the classification of homeomorphisms of 2--manifolds and the classification of 3--manifolds.}
Acta Math. 142 (1979), no. 1-2, 123--155. 

\bibitem{Hougardy10SurfRealIntersSegFunc}
S.~Hougardy, F.~H. Lutz, and M.~Zelke.
\emph{Surface realization with the intersection segment functional}.
Experiment. Math., {\bf 19}, (2010), no. 1, 79--92.

\bibitem{JO} W.~Jaco and U.~Oertel.
\emph{An algorithm to decide if a 3-manifold is a Haken manifold.}
Topology 23 (1984), no. 2, 195--209. 

\bibitem{JR} W.~Jaco and J.~H.~Rubinstein.
\emph{0--efficient triangulations of 3--manifolds}.
Journal of Differential Geometry {\bf 65}, (2003), no. 1, 61--168.

\bibitem{JR:Inflate} W.~Jaco and J.~.H.~Rubinstein.
\emph{Inflations of ideal triangulations}. Advances in Mathematics {\bf 267}, (2014) 176--224.

\bibitem{JRT} W.~Jaco, J.~H.~Rubinstein and S.~Tillmann. 
\emph{Minimal triangulations for an infinite family of lens spaces}.
Journal of Topology {\bf 2}, (2009), no. 1, 157--180.

\bibitem{JRT-2} W.~Jaco, J.~H.~Rubinstein and S.~Tillmann.
\emph{Coverings and minimal triangulations of $3$-manifolds}.
Algebr. Geom. Topol., {\bf 11}, (2011), no. 3, 1257--1265.

\bibitem{Kalelkar08ChiQuadrNS} T.~Kalelkar.
\emph{Euler characteristic and quadrilaterals of normal surfaces}.
Proc. Indian Acad. Sci. Math. Sci., {\bf 118}, (2008), no. 2, 227--233.

\bibitem{Kneser29ClosedSurfIn3Mflds}
H.~Kneser.
\emph{Geschlossene Fl\"{a}chen in dreidimensionalen Mannigfaltigkeiten}.
Jahresbericht der deutschen Mathematiker-Vereinigung, {\bf 38}, (1929),
248--260.

\bibitem{LT2013} F.~Luo and S.~Tillmann: \emph{Triangulated 3-Manifolds: Combinatorics and Volume Optimisation}, arXiv:1312.5087.

\bibitem{Lutz08EnumRandRealTrigSurf}
F.~H.~Lutz.
\emph{Enumeration and random realization of triangulated surfaces}.
In {\em Discrete differential geometry}, volume~38 of 
Oberwolfach Semin., Birkh\"auser, Basel, (2008), 235--253.

\bibitem{Lutz08ManifoldPage}
F.~H. Lutz.
\newblock {T}he {M}anifold {P}age.
\newblock {\url{http://www.math.tu-berlin.de/diskregeom/stellar}}.

\bibitem{Mat1990} S.~V.~Matveev. \emph{Complexity theory of three-dimensional manifolds}.
Acta Appl. Math. {\bf 19}, (1990), no. 2, 101--130.

\bibitem{Mat2003} S.~Matveev: \emph{Algorithmic topology and classification of 3-manifolds.} Algorithms and Computation in Mathematics, 9. Springer-Verlag, Berlin, 2003.


\bibitem{McMullen83Pol2MfldsOfHighGenus}
P.~McMullen, C.~Schulz, and J.~M.~Wills.
\emph{Polyhedral $2$-manifolds in $E^{3}$ with unusually large genus}.
Israel J. Math., {\bf 46}, (1983), no. 1--2, 127--144.

\bibitem{Ringel74MapColThm}
G.~Ringel. \emph{Map color theorem}. Volume 209 of \emph{Die Grundlehren der
mathematischen Wissenschaften}. Springer-Verlag, New York, (1974).


\bibitem{Rubinstein953SphereRec}
J.~H.~Rubinstein. \emph{An algorithm to recognize the 3--sphere}. In \emph{Proceedings of
the International Congress of Mathematicians (Z\"urich, 1994)}, volume 1, Birkh\"auser,
Basel, (1995), 601--611.

\bibitem{Rubinstein97PolMinSurf}
J.~H.~Rubinstein. \emph{Polyhedral minimal surfaces, Heegaard splittings and decision prob-
lems for 3–dimensional manifolds}. In \emph{Geometric Topology (Athens, GA, 1993)}.
AMS/IP Stud. Adv. Math. 2, Amer. Math. Soc. (1997), 1--20.

\bibitem{Schewe10NonrealMinTrigGenus6Surf}
L.~Schewe. \emph{Nonrealizable minimal vertex triangulations of surfaces: showing
nonrealizability using oriented matroids and satisfiability solvers}.
Discrete Comput. Geom., {\bf 43}, (2010), no. 2, 289--302.

\bibitem{Spreer10NormSurfsCombSlic}
J.~Spreer. \emph{Normal surfaces as combinatorial slicings}.
Discrete Math., {\bf 311}, (2011), no. 14, 1295--1309. 

\bibitem{Steinitz06UberEulPolyederRel}
E.~Steinitz. \emph{\"Uber die Eulerschen Polyederrelationen}.
Arch. der Math. u. Phys., {\bf 11}, (1906), no. 3, 86--88.

\bibitem{Thompson943SphereRec}
A.~Thompson. \emph{Thin position and the recognition 
problem for $S^3$}, Math. Res. Lett. 1, (1994), 613--630.

\bibitem{ti} S.~Tillmann. \emph{Normal surfaces in 
topologically finite 3--manifolds}, L'Enseignement 
Math\'ematique, {\bf 54}, (2008), no. 2, 329--380.

\bibitem{Tollefson98QuadTheory} J.~L.~Tollefson.
\emph{Normal surface $Q$--theory}, 
Pacific J. of Math., {\bf 183}, (1998), 359--374.

\bibitem{Ziegler08PolSurfsOfHighGenus}
G.~M.~Ziegler. \emph{Polyhedral surfaces of high genus}.
In {\em Discrete differential geometry}, volume~38 of 
Oberwolfach Semin., Birkh\"auser, Basel, (2008), 191--213.

\end{thebibliography}
\end{document}